\documentclass[11pt,onecolumn,draftclsnofoot]{IEEEtran}

\usepackage{tikz}
\usetikzlibrary{arrows,chains,matrix,positioning,scopes}
\makeatletter
\tikzset{join/.code=\tikzset{after node path={%
\ifx\tikzchainprevious\pgfutil@empty\else(\tikzchainprevious)%
edge[every join]#1(\tikzchaincurrent)\fi}}}
\makeatother

\tikzset{>=stealth',every on chain/.append style={join},
         every join/.style={->}}
\tikzstyle{labeled}=[execute at begin node=$\scriptstyle,
   execute at end node=$]
\usepackage{cite}
\usepackage{array}

\usepackage{rotating}
\usepackage{graphicx}
\usepackage{caption}
\usepackage{subcaption}
\usepackage{tablefootnote}

\usepackage{lipsum, color}
\usepackage{epstopdf}
\usepackage{arydshln}
\usepackage{float}
\usepackage{wrapfig}
\usepackage{empheq}
\usepackage{hyperref}
\usepackage{amssymb, amsfonts}
\usepackage{amsmath}
\usepackage{amsthm}

\newtheorem{assumption}{Assumption}
\newtheorem{example}{Example}
\newtheorem{remark}{Remark}
\newtheorem{definition}{Definition}
\newtheorem{theorem}{Theorem}
\newtheorem{lemma}{Lemma}
\newtheorem{proposition}{Proposition}

\usepackage[noend]{algpseudocode}

\usepackage[multiple]{footmisc}

\title{On the Absence of Spurious Local Trajectories in {Time-varying} Nonconvex Optimization}

\author{Salar Fattahi, Cedric Josz, Yuhao Ding, Reza Mohammadi, Javad Lavaei, and Somayeh Sojoudi}

\usepackage{soul}

\begin{document}
\maketitle

\renewcommand{\thefootnote}{\fnsymbol{footnote}}

\footnotetext{A shorter version of this paper has been submitted for conference publication~\cite{online_spurious_tech_2020_2}. Sections \ref{sec:dyn_mat}, \ref{sec:theo_found} and \ref{sec:sys_Jac} are the new additions to the paper.}

\footnotetext{Salar Fattahi is with University of Michigan, Ann Arbor (email: {fattahi@umich.edu}). Cedric Josz is with Columbia University (email: {cj2638@columbia.edu}).  Yuhao Ding, Reza Mohammadi, Javad Lavaei, and Somayeh Sojoudi are with the University of California, Berkeley (email: {yuhao\_ding@berkeley.edu}, {mohammadi@berkeley.edu}, {lavaei@berkeley.edu}, {sojoudi@berkeley.edu}). }

\renewcommand{\thefootnote}{\arabic{footnote}}


\begin{abstract}
In this paper, we study the landscape of an online nonconvex optimization problem, for which the input data vary over time and the solution is a trajectory  rather than a single point.  To understand the complexity of finding a global solution of this problem,  we introduce the notion of \textit{spurious (i.e., non-global)   local trajectory} as a generalization to the notion of spurious local solution in nonconvex (time-invariant) optimization. We develop an ordinary differential equation (ODE) associated with a time-varying nonlinear dynamical system which, at limit, characterizes the spurious local solutions of the time-varying optimization problem. We prove that the absence of spurious local trajectory is closely related to the transient behavior of the developed system. In particular, we show that if the problem is time-varying,  the data variation may force all of the ODE trajectories initialized at arbitrary local minima at the initial time to  gradually converge to the global solution trajectory.  We study the Jacobian of the dynamical system along  a local minimum trajectory and show how its eigenvalues are manipulated by the natural data variation in the problem, which may consequently trigger escaping poor local minima over time. 
\end{abstract}

%

\section{Introduction}

Sequential decision making with time-varying data is at the core of most of today's problems. For example, the optimal power flow (OPF) problem in the electrical grid should be solved every 5 minutes in order to match the supply of electricity with a demand profile that changes {over time~\cite{low2014convex}}. {Other examples include the training of {dynamic} neural networks~\cite{gupta2004static}, dynamic matrix recovery~\cite{xu2016dynamic, xu2017simultaneous}, time-varying multi-armed bandit problem~\cite{zeng2016online}, robot navigation and
obstacle avoidance \cite{belkhouche2007autonomous}, and many other applications \cite{simonetto2020time}}. Indeed, most of  these problems are large-scale and should be solved in real-time, which strongly motivates the need for fast algorithms in such optimization frameworks.

A recent line of work has shown that a surprisingly large class of data-driven and nonconvex optimization problems---including matrix completion/sensing, phase retrieval, and dictionary learning, robust principal component analysis---has a \textit{benign landscape}, i.e., every local solution is also global~\cite{bhojanapalli2016global, ge2016matrix, zhang2019sharp, fattahi2018exact}. A local solution that is not globally optimal is called {\it spurious}. At the crux of the results on the absence of spurious local minima is the assumption on the static and time-invariant nature of the optimization. Yet, in practice, many real-world and data-driven problems are time-varying and require online optimization. This observation naturally gives rise to the following question: 

\textit{Would fast local-search algorithms escape spurious local minima in online nonconvex optimization, similar to their time-invariant counterparts?}

In this paper, we attempt to address this question by developing a  control-theoretic framework for analyzing the landscape of online and time-varying optimization.
In particular, we demonstrate that even if a time-varying optimization may have undesired point-wise local minima at almost all times, the variation of its landscape over time would enable simple local-search algorithms to escape these spurious local minima. 
Inspired by this observation, this paper provides a new machinery to analyze the global landscape of online decision-making problems by drawing tools from optimization and control theory.

We consider a class of nonconvex and online optimization problems where the input data vary over time. First, we introduce the notion of \textit{spurious local trajectory} as a generalization to the point-wise spurious local solutions. We show that a time-varying optimization can have point-wise spurious local minima at every time step, and yet, it can be free of spurious local trajectories. By building upon this notion, we consider a general class of nonconvex optimization problems and model their local trajectories via an  ordinary differential equation (ODE)  representing a time-varying nonlinear dynamical system. We show that the absence of the spurious local trajectories in this time-varying optimization is equivalent to the convergence of all solutions in its corresponding ODE. Based on this equivalence, {we analyze different classes of time-varying optimization problems and present sufficient conditions under which, despite possibly having point-wise spurious local minima at all times, the time-varying problem is free of spurious local trajectories.} This implies that the time-varying nature of the problem is essential for the absence of spurious local trajectories. Finally, we analyze the Jacobian of the ODE along a local minimum trajectory and show how its eigenvalues are manipulated by the data variation. 

\subsection{Related Works}

\textbf{Benign landscape:} Nonconvexity is inherent to many problems in machine learning; from the classical compressive sensing and matrix completion/sensing~\cite{donoho2006most,candes2010matrix,candes2009exact} to the more recent problems on the training of deep neural networks~\cite{li2018visualizing}, they often possess nonconvex landscapes. Reminiscent from the classical complexity theory, this nonconvexity is perceived to be the main contributor to the intractability of these problems. In many (albeit not all) cases, this intractability implies that in the worst-case instances of the problem, spurious local minima exist and there is no efficient algorithm capable of escaping them. However, a lingering question remains unanswered: are these worst-case instances common in practice or do they correspond to some pathological or rare cases?

Answering this question has been the subject of many recent studies. In particular, it has been shown that nearly-isotropic classes of problems in matrix completion/sensing \cite{bhojanapalli2016global, ge2016matrix, ge2017no}, robust principle component analysis \cite{fattahi2018exact, josz2018theory}, and dictionary recovery \cite{sun2017complete} have {benign landscape}, implying that they are free of spurious local minima.  It has also been proven recently in~\cite{pmlrv80kleinberg18a} that under some conditions, the stochastic gradient descent may escape the sharp local minima in the landscape. 
At the core of the aforementioned results is the assumption on the static and time-invariant nature of the landscape. In contrast, many real-world problems should be solved sequentially over time with time-varying input data. For instance, in the optimal power flow problem, the electricity consumption of the consumers changes hourly~\cite{tang2017,tang2019bis}. Therefore, it is natural to study the landscape of such time-varying nonconvex optimization problems, by taking into account their dynamic nature.

\vspace{2mm}

%

\textbf{Time-varying dynamical systems:} Recently, there has been a growing interest in analyzing the performance of numerical algorithms from a control-theoretical perspective~\cite{su2014differential, wibisono2016variational, zhang2018direct, chen2018neural, scieur2017,xu2018}. Roughly speaking, the general idea behind these approaches is to analyze the convergence of a specific algorithm by first modeling its limiting behavior as an autonomous (time-invariant) ODE that describes the evolution of a dynamical system, and then studying its stability properties. As a natural extension, one would generalize this approach to online optimization by modeling its limiting behavior as a non-autonomous ODE corresponding to a time-varying dynamical system. However, the stability analysis of time-varying dynamical systems is highly convoluted in the nonlinear case. We note that several necessary and sufficient conditions for the stability of linear time-varying systems were proposed  in \cite{zhou2016}. A generalized time-varying Lyapunov function was proposed in \cite{aeyels1998} and has been applied in \cite{teel1999} to study the stability of an averaged system. Furthermore, slowly time-varying systems are investigated in \cite{dacunha2007}.

\section{Case Studies}

In this section, we present  empirical studies on the dynamic landscapes of two problems in power systems and machine learning:  optimal power flow and dynamic matrix recovery.
	\subsection{Electrical Power Systems}
\label{sec:power}

\begin{figure*}[ht]
    \centering
 \begin{subfigure}[b]{0.43\textwidth}
            \centering
            \includegraphics[width=\textwidth]{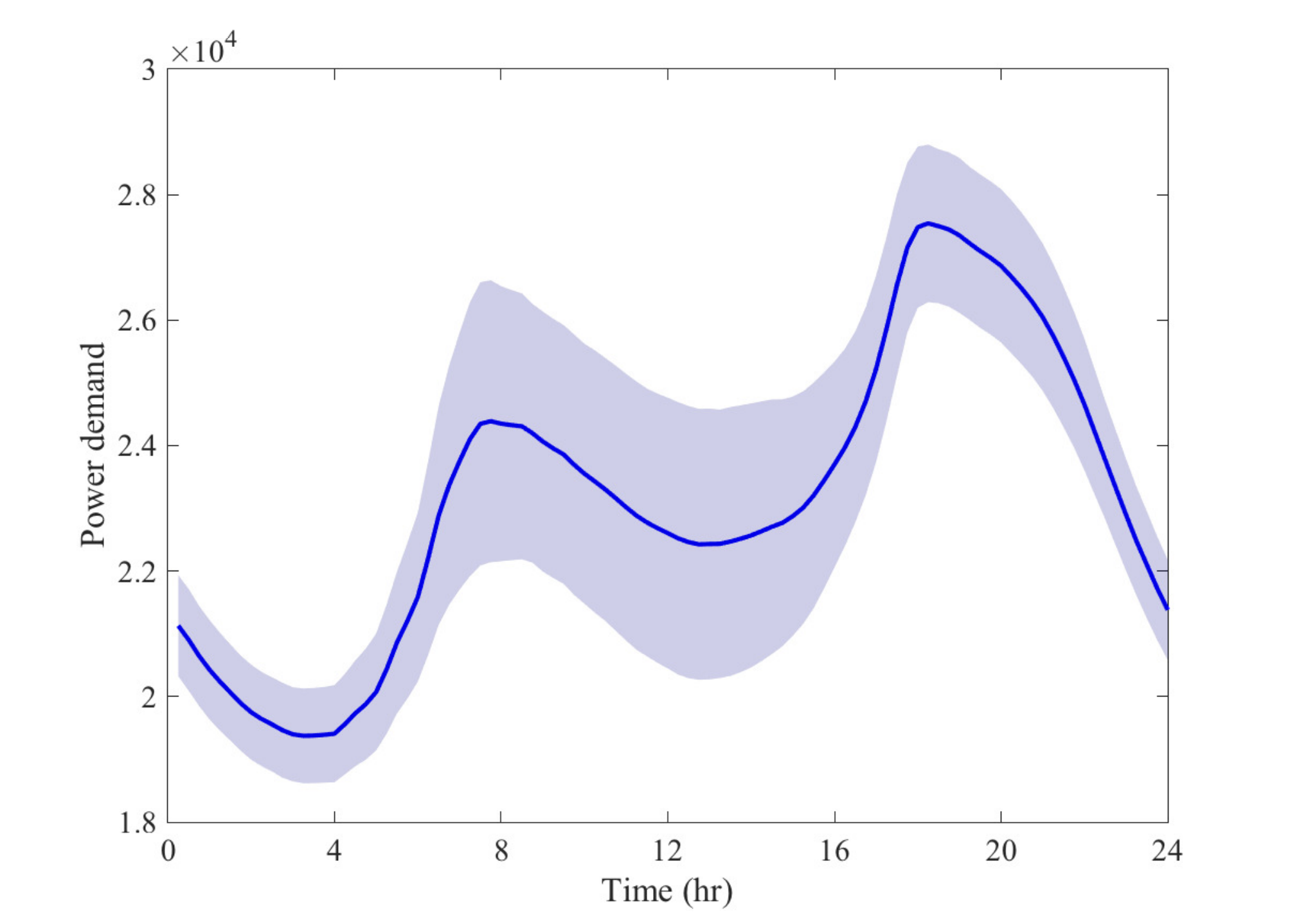}
             \caption{California average load profile for January 2019.}
    \label{fig_CA_load_profile}
    \end{subfigure}\hspace{2mm}
\begin{subfigure}[b]{0.43\textwidth}
            \centering
            \includegraphics[width=\textwidth]{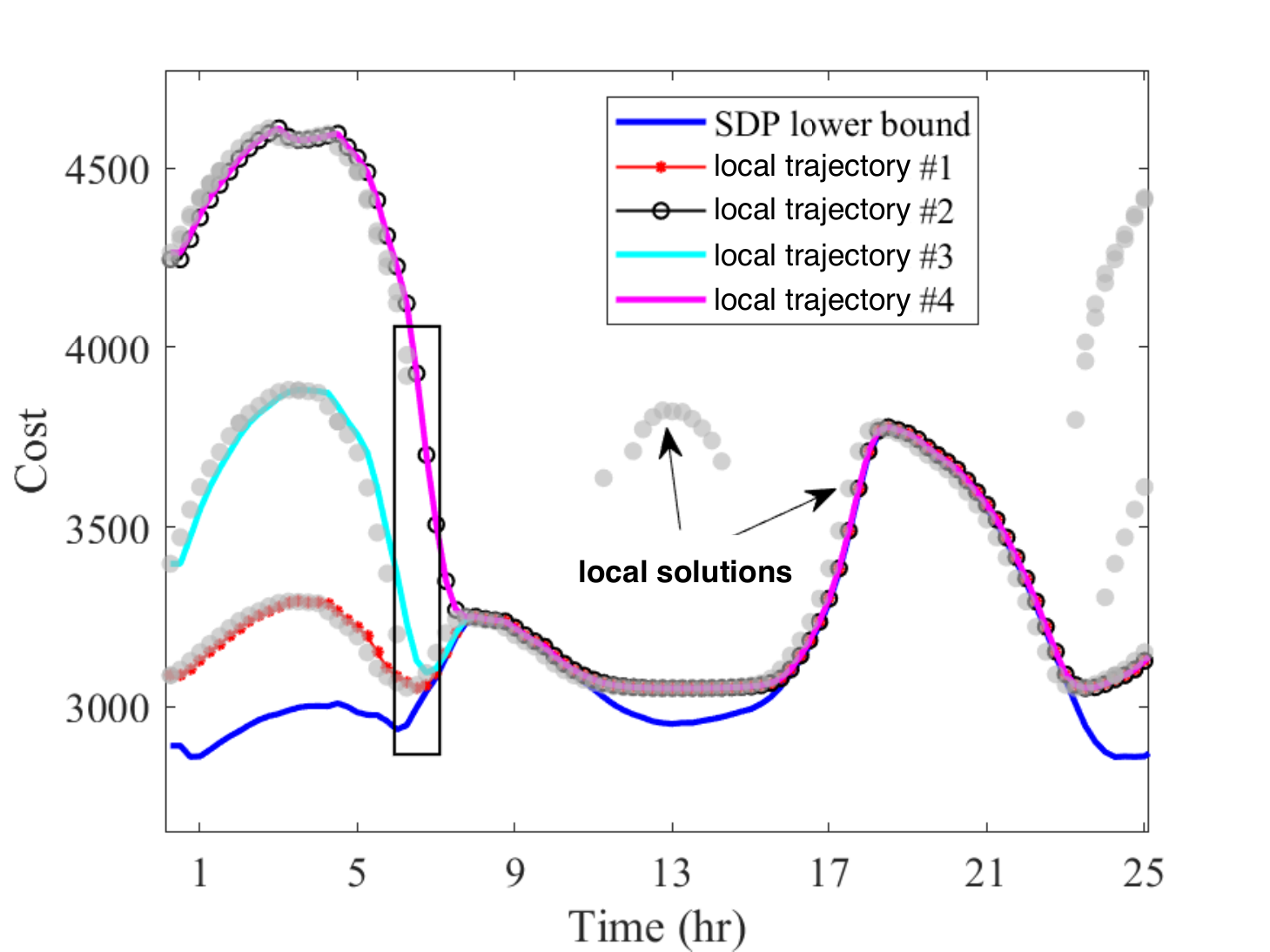}
             \caption{Solution trajectories of  time-varying optimal power flow.}
    \label{fig_CA_periodic_delta005}
    \end{subfigure}
    \caption{Case study in power systems (data collected from \url{http://www.caiso.com}).}
    \end{figure*}


In the optimal power flow problem, the goal is to match the supply of electricity with a time-varying demand profile, while satisfying the network, physical, and technological constraints. In practice, the problem is solved sequentially over time with the constraint that at every time-step, the solution cannot be significantly different from the one obtained in the previous time-step due to the so-called ramping constraints of the generators. We consider the IEEE 9-bus system~\cite{localSol_5bus} and initialize the system from the global solution, as well as three different spurious local solutions. We then change the load over time based on the California average load profile for the month of January 2019 (Figure~\ref{fig_CA_load_profile}). The optimal power flow problem is then solved sequentially using local search every 15 minutes for the period of 24 hours, while taking into account the temporal couplings between solutions via the ramping constraints. The trajectories of the solutions for the optimal power flow problem with different initial points appear in Figure~\ref{fig_CA_periodic_delta005}. 
	In this figure, the solid blue line represents the cost obtained by the semidefinite programming (SDP) relaxation of the optimal power flow~\cite{lavaei1_2}. This curve is a lower bound to the globally optimal cost and serves as a certificate of the global optimality whenever it touches other trajectories.

The gray circles in these plots are some of the local solutions that were obtained via a Monte Carlo simulation.
Based on  Figure~\ref{fig_CA_periodic_delta005}, indeed there exist multiple local solutions at almost all time-step (some of them emerge over time). However, surprisingly, the trajectories of the local solutions that are initialized at different points all converge towards the global solution.  This implies that there is no spurious local trajectory, and therefore local search methods are able to find global minima of the optimal power flow problem at future times even when they start from poor local minima at the initial time. 

\subsection{Dynamic Matrix Recovery}
\label{sec:dyn_mat}

In the dynamic matrix recovery problem, the goal is to recover a time-varying low-rank matrix, based on a limited number of  linear observations~\cite{xu2016dynamic, xu2017simultaneous}. 
This problem can be formulated as follows:
\begin{align}\label{eq:mat}
\inf_{X\in\mathbb{R}^{n\times r}}\sum_{i=1}^{m}\left(\langle A_i, XX^\top\rangle-d_{i}(t)\right)^2
\end{align}
where $\langle \cdot, \cdot\rangle$ is the inner product operator, $\{A_i\}_{i=1}^m$ are the sensing matrices, and $d(t)$ is the time-varying measurements vector. Equivalently,~\eqref{eq:mat} can be re-written as
\begin{align}\label{mat_sens}
\inf_{X\in\mathbb{R}^{n\times r}, \epsilon\in\mathbb{R}^m}& \sum_{i=1}^{m} \epsilon_i^2\nonumber\\
\mathrm{s.t.}~~~~~& \langle A_i, XX^\top\rangle-\epsilon_i = d_{i}(t) ~,~~ i = 1,\hdots,m
\end{align}
Assuming that $d(t)$ does not change over time, it is well-known that the above optimization problem has no spurious local minima if the sensing matrices $\{A_i\}_{i=1}^m$ satisfy a certain \textit{restricted isometry property} (RIP). In particular, it is said that the sensing matrices $\{A_i\}_{i=1}^m$ satisfy RIP with a constant $\delta\in [0,1)$ if the inequality $(1-\delta)\|X\|_F^2\leq \frac{1}{m}\sum_{i=1}^{m}\langle A_i, X\rangle\leq (1+\delta)\|X\|_F^2$ is satisfied for every $X\in\mathbb{R}^{n\times n}$ whose rank is upper bounded by $2r$ ($\|X\|_F$ is the Frobenious norm of the matrix $X$). Recently,~\cite{zhang2019sharp} showed that if $r = 1$, an RIP constant of $\delta<1/2$ is both necessary and sufficient for the benign landscape of the time-invariant matrix recovery problem.

Consider the  sensing matrices
\begin{align}\small
	A_1 \!&=\! \begin{bmatrix}
	1 \!&\! 0\\
	0 \!&\! \frac{1}{2}
	\end{bmatrix}\!,\quad \quad \ \ A_2 \!=\! \begin{bmatrix}
	0 \!\!&\!\! \frac{\sqrt{3}}{2}\\
	\frac{\sqrt{3}}{2} \!\!&\!\! 0
	\end{bmatrix}\\
	A_3 \!&=\! \begin{bmatrix}
	1 \!\!\!&\!\!\! -\frac{1}{\sqrt{2}}\\
	\frac{1}{\sqrt{2}} \!\!\!&\!\!\! 0
	\end{bmatrix}\!,\quad  A_4 \!=\! \begin{bmatrix}
	0 \!\!&\!\! 0\\
	0 \!\!&\!\! \frac{\sqrt{3}}{2}
	\end{bmatrix}\nonumber
\end{align}
with the time-invariant measurement vector $d = \begin{bmatrix}
1 & 0 & 0 & 0
\end{bmatrix}^\top$ and $r=1$.
The paper~\cite{zhang2019sharp} proved that the RIP constant for the above sensing matrices is equal to $1/2$. This implies that the matrix recovery problem with the aforementioned sensing matrices is prone to having spurious local minima.
In fact,~\cite{zhang2019sharp} showed that the above problem has one global solution at $Z = \begin{bmatrix}
1 & 0
\end{bmatrix}^\top$ and one spurious local solution at $X = \begin{bmatrix}
0, {1}/{\sqrt{2}}
\end{bmatrix}^\top$. 
Now, consider the time-varying version of the above instance, where the measurement vector changes over time, as in:
$$d(t) = \begin{bmatrix}
(0.8+0.2\cos t)^2+\frac{1}{2}(0.2\sin t)^2\\
\sqrt{3}(0.2\sin t)(0.8+0.2\cos t)\\
0\\
\frac{\sqrt{3}}{2}(0.2\sin t)^2
\end{bmatrix}$$
It is easy to see that $Z = \begin{bmatrix}
0.8+0.2\cos t & 0.2\sin t
\end{bmatrix}^\top$ is the trajectory of the globally optimal solution to the defined dynamic matrix recovery problem. Moreover, using a Gradient descent algorithm initialized at the spurious local solution at time $t=0$, we solve~\eqref{mat_sens} sequentially over time with an appropriate regularization (to be defined later). Figures~\ref{fig_trajectories} and~\ref{fig_cost} show that, despite the fact that the problem has a spurious local minimum at $t=0$ and future times, its local trajectory gradually converges to the global one.

\begin{figure*}
      \centering
      \subcaptionbox{The trajectories of local and global solutions over time\label{fig_trajectories}}
        {  \includegraphics[width=0.36\linewidth]{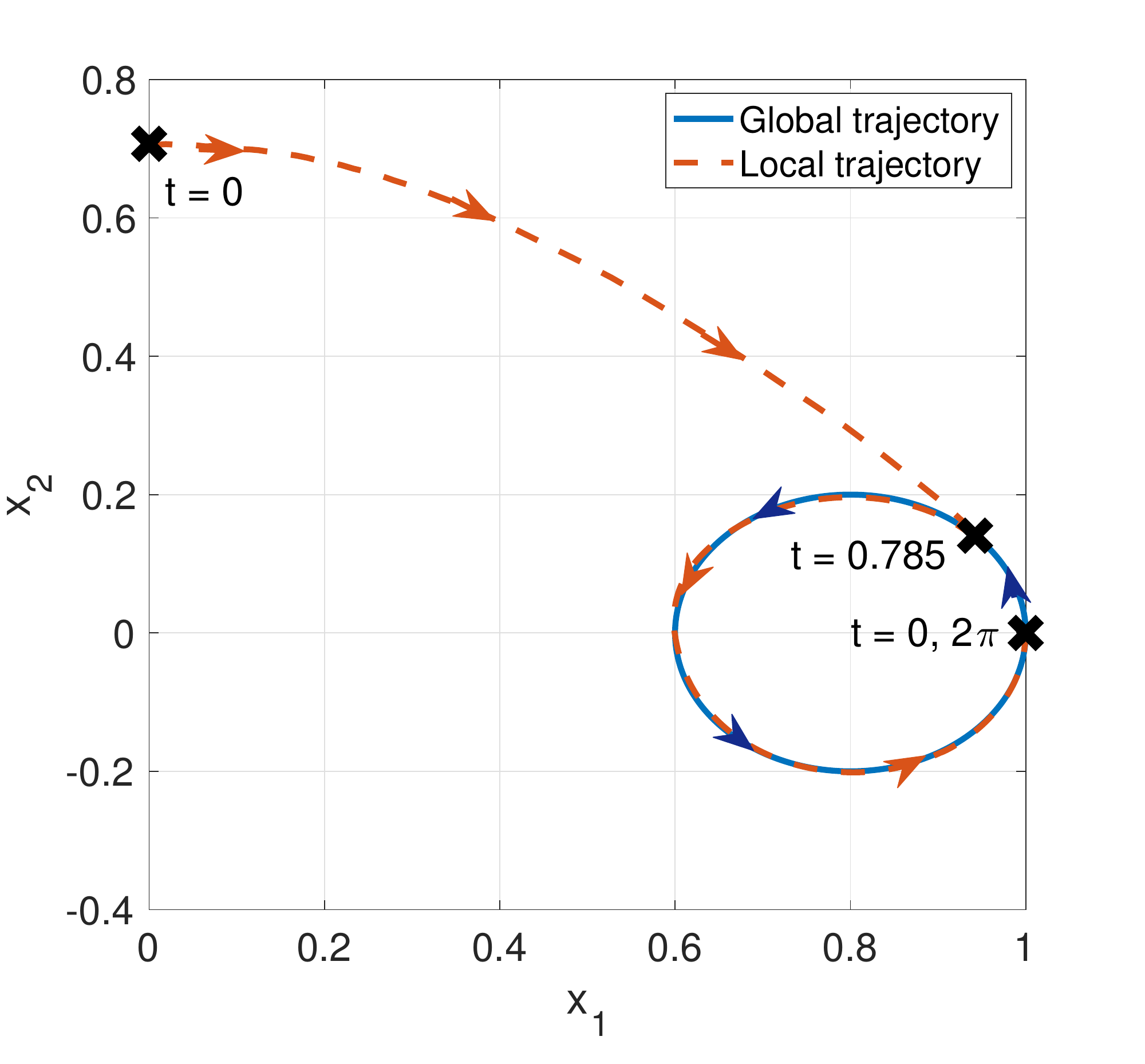}}\hspace{1mm}
      \subcaptionbox{The objective value of the local trajectory over time\label{fig_cost}}
        { \includegraphics[width=0.38\linewidth]{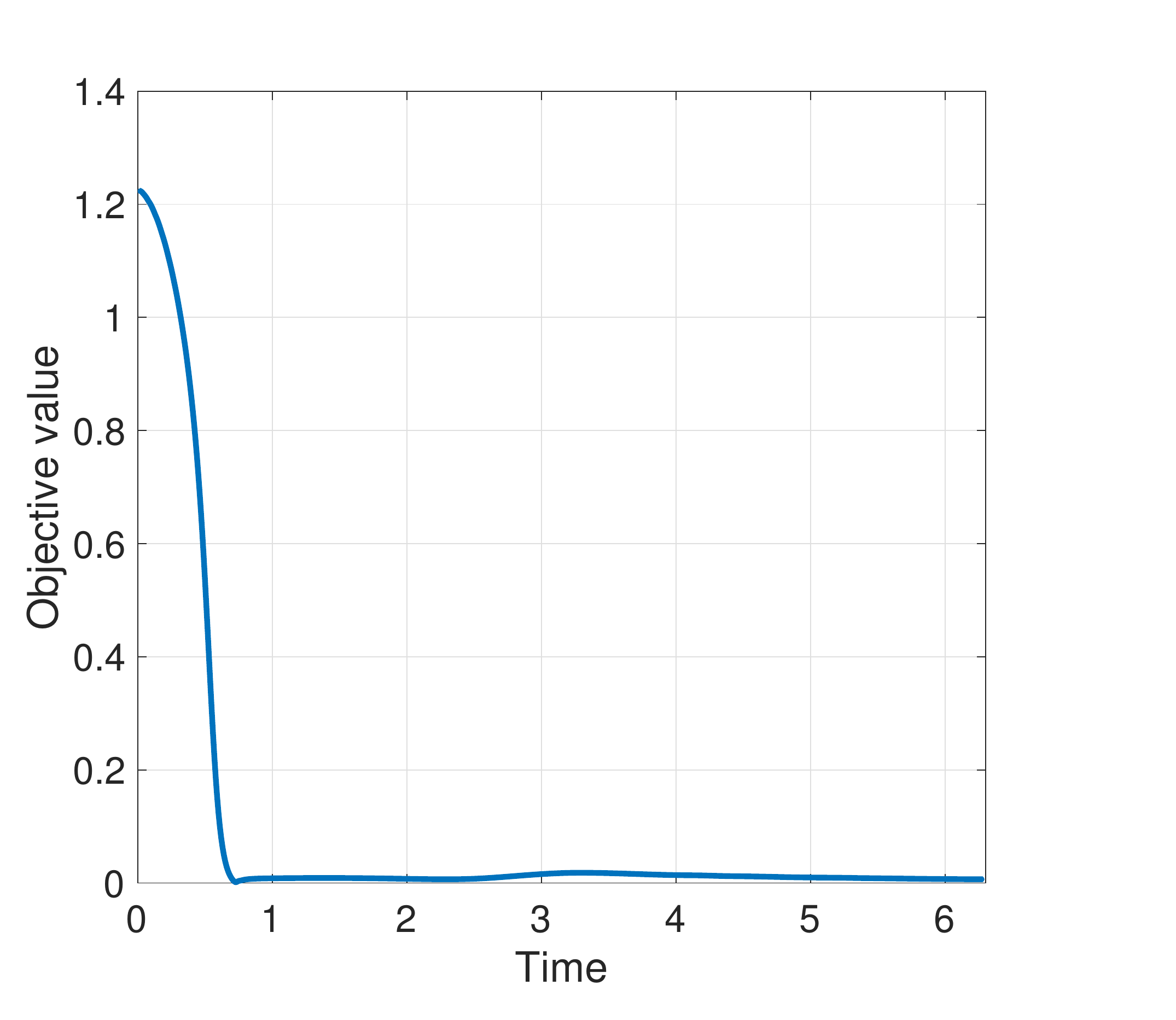}}
      \caption{Case study in matrix recovery.}
    \end{figure*}


\section{Notion of Spurious Local Trajectory}
\label{sec:Notion of spurious local trajectory}
Inspired by the above case studies, we consider the effect of the variation in the input data on the landscape of the optimization problem. 
We focus on the following time-varying nonconvex optimization:
\begin{equation}
\inf_{{x(t)}\in \mathbb{R}^n} ~ f({x(t)},t) ~~~\text{s.t.}~~~ h_i({x(t)}) = d_i(t),~ i = 1,\hdots,m 
\label{eq:problem}
\end{equation}
where the objective function $f({x(t)},t)$ and the right-hand side of the equality constraints vary over time $t\in [0,T]$. {We assume that $f: \mathbb{R}^n \times [0,T] \longrightarrow \mathbb{R}$ is a continuously differentiable function}. Moreover, $h_i: \mathbb{R}^n \longrightarrow \mathbb{R}$ and $d_i:  [0,T] \longrightarrow \mathbb{R}$ for $i=1,\hdots,m$ are  twice continuously differentiable functions, and that $T>0$ is a finite time horizon. 
Moreover, we assume that $f$ is uniformly bounded from below (i.e., $f({x(t)},t)\geq M$ for some constant $M$) and that the problem is feasible for all $t\in[0,T]$. {The objective function $f(x,t)$ may be nonconvex in $x$ and the constraint function $h(x)=(h_1(x),\ldots,h_m(x))$ may be nonlinear in $x$.}
Note that the dynamic matrix recovery problem~\eqref{mat_sens} is a special case of~\eqref{eq:problem}.
\begin{remark}
	 Inequality constraints can also be included in~\eqref{eq:problem} through a reformulation technique. In particular, suppose that~\eqref{eq:problem} includes a set of inequality constraints $g_j(x)\leq v_j(t)$ for $j = 1,\dots, l$. Then, one can reformulate them as equality constraints through the following procedure:
	\begin{itemize}
		\item[1.]  Rewrite the inequality constraints by introducing a slack variable $s\in\mathbb{R}^l$, as in
		{$$g_j(x(t))+ s_j(t) = v_j(t),~~~~ j = 1,\dots, l$$}
		\item[2.] Augment the objective function with a penalty {$p(s(t)) = \sum_{j=1}^{l}p_j(s_j(t))$}.
	\end{itemize}
Here, {$p_j(s_j(t))$} are nonsmooth loss functions for an exact reformulation. Furthermore, they can be relaxed to continuously differentiable loss functions at the expense of incurring some (controllable) approximation errors; see~\cite{zangwill1967non, bertsekas1997nonlinear}. This implies that the previously-introduced optimal power flow problem can be reformulated as~\eqref{eq:problem}.
\end{remark}

In practice, one can only hope to sequentially solve this problem at discrete times $0 = t_0 < t_1 < t_2 < \hdots < t_N =  T$. However, notice that~\eqref{eq:problem} is un-regularized. In particular, depending on the properties of the objective function, an arbitrary solution to~\eqref{eq:problem} at time $t_k$ can be arbitrarily far from that of~\eqref{eq:problem} at time $t_{k-1}$. However---as elucidated in our case study on the optimal power flow problem--- it is neither practical nor realistic to have solutions that change abruptly over time in many real-world problems. One way to circumvent this issue is to regularize the problem at time $t_{k+1}$ by penalizing the deviation of its solution from the one obtained at time $t_{k}$. Precisely, we employ a quadratic proximal regularization as is done in online learning~\cite{do2009}. 

\begin{definition}
	\label{def:discrete}
	Given evenly spaced-out time steps $0 = t_0 < t_1 < t_2 < \hdots < t_N = T$ for some integer $N$, a sequence $x_0,x_1,x_2, \hdots, x_N$ is said to be a \textbf{discrete local trajectory} of the time-varying optimization \eqref{eq:problem} if the following holds: 
	\begin{enumerate}
		\item $x_0$ is a local solution to the time-varying optimization \eqref{eq:problem} at time $t_0=0$;
		\item for $k=0,1,2, \hdots,N-1$, $x_{k+1}$ is local solution to the regularized problem
		\begin{align}
		\begin{array}{ll}
		\inf_{x\in \mathbb{R}^n} & f(x,t_{k+1}) ~+~  \alpha ~\frac{\|x - x_k\|^2}{2(t_{k+1}-t_k)} \label{obj_reg}\\[.3cm]
		\text{s.t.} & h_i(x) = d_i(t_{k+1}) ~,~~~ i = 1,\hdots,m.
		\end{array}
		\end{align}
	\end{enumerate}
	Above, $\alpha>0$ is a fixed regularization parameter and $\| \cdot \|$ denotes the Euclidian norm.
\end{definition}
Note that in the above definition, the term \textit{local solution} refers to any feasible point that satisfies the Karush-Kuhn-Tucker (KKT)  conditions for~\eqref{obj_reg}. A natural approach to characterizing the global landscape of~\eqref{eq:problem} is to analyze discrete local trajectories of the regularized problem~\eqref{obj_reg}. However, notice that the non-convexity of~\eqref{obj_reg} may lead to \textit{bifurcations} in discrete local trajectories. In particular, given a local solution $x_k$, the regularized problem~\eqref{obj_reg} may possess two local solutions $x^{(1)}_{k+1}$ and $x^{(2)}_{k+1}$, each resulting in a different discrete local trajectory.\footnote{For example, there exist two discrete trajectories starting at $x_0=0$ and at time $t_0=0$ for the time-varying objective function $f(x,t) := x^2(T/2-t)$. Indeed, the discrete trajectory stays at $x_k=0$ for $t_k\leq T/2$ and then, due to the regularization, it bifurcates into two separate discrete trajectories.} {The non-uniqueness of the discrete local trajectories due to the bifurcation will make the analysis inconclusive. This is because the next solution of the problem given the current solution is not well-defined and due to the number of possibilities at each step, the solution trajectory is not unique and can take an exponential number of possibilities depending on the settings of the numerical algorithm (the choice of descent directions and step sizes). However,} in what follows, we show that such {bifurcations} disappear in the {ideal} scenario, where the regularized problem can be {sampled} arbitrarily fast, or equivalently, as we increase $N$ to infinity. In particular, given a fixed initial local solution $x_0$, we show that any discrete local trajectory starting from $x_0$ converges uniformly to the unique solution to a well-defined ODE that is initialized at $x_0$. By building upon this result, we introduce the notion of spurious local trajectory as a generalization to the notion of spurious local minima.

Given an initial local solution $x_0$, consider the following initial value problem:
\begin{subequations}
	\begin{align}
	\dot{x} &= -\frac{1}{\alpha} \eta(x,t) + \theta(x)\dot{d}\\
	x(0) &= x_0
	\end{align}  
	\label{eq:ode}
\end{subequations}
where
\begin{subequations}
\begin{align}
\eta(x,t) := & ~~ \left[I-\mathcal{J}(x)^\top(\mathcal{J}(x)\mathcal{J}(x)^\top)^{-1}\mathcal{J}(x)\right]\nonumber\\
&~~~~~ \times \nabla_x f(x,t), \\
\theta(x) ~ := & ~~ \mathcal{J}(x)^\top(\mathcal{J}(x)\mathcal{J}(x)^\top)^{-1}.
\end{align}
\end{subequations}
Above, $\mathcal{J}(x)$ denotes the Jacobian of the left-hand side of the constraints $h(x) = [h_1(x),\hdots,h_m(x)]^\top$ and $d(t)$ denotes the right-hand side of the constraints, that is to say $d(t) = [d_1(t),\hdots,d_m(t)]^\top$. The term $\theta(x)\dot{d}$ captures the effect of data variation in the dynamics, and the function $\eta(x,t)$ can be interpreted as the orthogonal projection of the gradient $\nabla_x f(x,t)$ on the Kernel of $\mathcal{J}(x)^\top$.

Later, we will show that the solution to~\eqref{eq:ode} exists, it is unique, and can be used to fully characterize the limiting behavior of every discrete local trajectory of the time-varying problem~\eqref{eq:problem}.
\begin{assumption}[Uniform Boundedness]\label{assum11}
	There exist constants $R_1>0$ and $R_2>0$ such that, for any discrete local trajectory $x_0,x_1,x_2,\hdots$, the parameter $\|x_k\|$ and the objective function of~\eqref{obj_reg} at $x_k$ are upper bounded by $R_1$ and $R_2$, respectively, for every $k\in\{0,1,2,\hdots, N\}$.
\end{assumption}
{Assumption~\ref{assum11} is mild, and can be guaranteed by requiring the feasible region to be compact. This can be ensured by adding constraints on the magnitude of the variables. For instance, in the time-varying OPF, the variables of the problem, i.e., active and reactive power, voltage magnitudes, and their angles, are restricted to bounded sets implied by the laws of physics and technological constraints on physical devices. It is worth noting that the main results of the paper do not depend on the explicit values of the constants $R_1$ and $R_2$.}
\begin{assumption}[Non-singularity]\label{assum22}
	There exists a constant $c>0$ such that, for any discrete local trajectory $x_0,x_1,x_2,\hdots$, it holds that $\sigma_{\min}(\mathcal{J}(x_k))\geqslant c$ for all $k\in\{0,1,2,\hdots\}$, where $\sigma_{\min}$ denotes the minimal singular value.
\end{assumption}

{Assumption \ref{assum22} implies that linear independence constraint qualification (LICQ) holds at every point of a discrete local trajectory, which in turn implies that the constraints are non-degenerate. The LICQ is a simple sufficient condition to guarantee the well-definedness of the KKT points~\cite{nocedal2006numerical}, and is the most standard assumption in the optimization literature~\cite{boyd2004convex, bertsekas1997nonlinear, luenberger1984linear}. Indeed, most of the off-the-shelf solvers, such as IPOPT~\cite{waechter2009introduction}, only converge to solutions that automatically satisfy LICQ.}
\begin{theorem}[Existence and Uniqueness]\label{thm:exist_unique} {Let Assumption \ref{assum11} and Assumption \ref{assum22} hold. Suppose that $x_0$ is an arbitrary local solution to the time-varying optimization~\eqref{eq:problem} at $t = 0$. Then, the ODE~\eqref{eq:ode} with the initial value condition $x(0) = x_0$ has a unique continuously differentiable solution $x:[0,T]\to\mathbb{R}^n$.}
\end{theorem}
{Theorem~\ref{thm:exist_unique} states that the proposed ODE is well-defined and has a unique solution, provided that its initial value is a local solution, i.e., it satisfies the KKT conditions for the original time-varying optimization problem. As will be shown later, this assumption is crucial and cannot be relaxed in general. Given the unique solution to the proposed ODE, the next theorem precisely characterizes its relationship to \textit{any} discrete local trajectory of~\eqref{obj_reg} starting at $x_0$.}
\begin{theorem}[Uniform Convergence]\label{thm:conv} Let Assumption \ref{assum11} and Assumption \ref{assum22} hold. If $x_0$ is a local solution to the time-varying optimization~\eqref{eq:problem} at $t = 0$, then any discrete local trajectory initialized at $x_0$ converges towards the solution $x:[0,T]\to\mathbb{R}^n$ with $x(0) = x_0$, in the sense that
	\begin{align}\label{converge}
	\lim\limits_{N\to+\infty}\sup_{0\leq k\leq N}\|x_k-x(t_k)\| = 0,
	\end{align}
	where $N$ is the number of points in the discrete local trajectories that are evenly spaced-out in time.
\end{theorem}

\textit{Sketch of the proofs.} The proofs for Theorems~\ref{thm:exist_unique} and~\ref{thm:conv} are quite involved and hence, they are deferred to the next section. In what follows, we provide the high-level ideas of our developed proof techniques. 
Note that most of the classical results on ordinary differential equations, namely the Picard-Lindel\"of Theorem \cite[Theorem 3.1]{coddington1955theory}, the Cauchy-Peano Theorem \cite[Theorem 1.2]{coddington1955theory}, and the Carath\'eodory Theorem \cite[Theorem 1.1]{coddington1955theory}, can only guarantee the existence of a solution in a {local region}, i.e., a neighborhood $[0,\tau]$ where $\tau < T$ is potentially very small. On the other hand, the global version of Picard-Lindel\"of Theorem only holds under a restrictive Lipschitz condition, which is not satisfied for~\eqref{eq:ode}. Instead, we take a different approach to prove existence and uniqueness of the solution to~\eqref{eq:ode} (Theorem~\ref{thm:exist_unique}). The proof consists of three general steps:
\begin{enumerate}
	\item By building upon the Arzel\`a-Ascoli Theorem, we show that, among all the discrete local trajectories that are initialized at $x_0$, there exists at least one that is uniformly convergent to a continuously differentiable function $y:[0,T]\to\mathbb{R}^n$.
	\item By fully characterizing the KKT points of~\eqref{obj_reg}, we prove that $y$ is a solution to~\eqref{eq:ode} when $N\to +\infty$. 
	\item The uniqueness of the solution is then proved by showing the existence of an open and connected set $\mathcal{D}$ such that the proposed ODE is locally Lipschitz continuous on $\mathcal{D}$ and $(y(t),t)\in\mathcal{D}$ for every $t\in[0,T]$. This, together with~\cite[Theorem 2.2]{coddington1955theory}, completes the proof of Theorem~\ref{thm:exist_unique}.
\end{enumerate}

Given the existence and uniqueness of the solution to~\eqref{eq:ode}, we show the correctness of Theorem~\ref{thm:conv} by making an extensive use of the so-called backward Euler method~\cite{butcher2016numerical}. In particular, we show that {all} of the discrete local trajectories converge to a discretized version of the solution to~\eqref{eq:ode} that is obtained by the backward Euler method. This, together with the existing convergence results on the backward Euler iterations, completes the proof of Theorem~\ref{thm:conv}.$\hfill\square$

Now that we have established the connection between the discrete local trajectories and their continuous limit, we naturally propose the following definition.
 \begin{definition}
	\label{def:continous}
	A continuously differentiable function ${x(t)}: [0,T] \longrightarrow \mathbb{R}^n$ is said to be a \textbf{continuous local trajectory} of the time-varying optimization \eqref{eq:problem} if the following holds:
	\begin{enumerate}
		\item $x(0)$ is a local solution to the time-varying optimization \eqref{eq:problem} at time $t=0$;
		\item ${x(t)}$ is a solution to~\eqref{eq:ode}.
	\end{enumerate}
\end{definition}
{
The next definition will be at the core of our subsequent definition of spurious local trajectories. 
\begin{definition}
The \textbf{region of attraction} of a local minimum $x^\ast(t)$ of $f(\cdot,t)$ {in the feasible set $\mathcal{F}(t)=\{x \in \mathbb{R}^n: h(x)=d(t)\}$} at a given time $t$ is defined as:
\begin{align*}
&\Big\{x_0\in \mathcal{F}(t) \ \big|\ \lim_{s \rightarrow \infty}\tilde x(s)=x^\ast(t) \quad \text{where} \quad \frac{d \tilde x(s)}{d s}= -\frac{1}{\alpha}\eta(\tilde x(s), t) +\theta(\tilde x (s))\dot{d}(t)  \quad \text{and} \quad \tilde x(0)=x_0 \Big\}
\end{align*}
\end{definition}}
We next introduce the central notion in this paper. 
{
\begin{definition}\label{def:spurious}
	A continuous local trajectory $x(t)$ is said to be \textbf{spurious} if for all $\bar{T} < T$, there exists a time $t \in [\bar{T},T]$ such that $x(t)$ does not belong to the region of attraction of a global solution of $f(\cdot,t)$. Accordingly, the time-varying optimization problem \eqref{eq:problem}  is said to \textbf{have no spurious local trajectories} if, when initialized at a local solution, any continuous local trajectory $x(t)$ belongs to the region of attraction of a global solution of $f(\cdot,t)$ at all times $t\in[\bar T,T]$ for some constant $\bar{T}<T$.
\end{definition}}

{So far, we have taken the time horizon $T$ to be finite. However, the above definition naturally applies to problems with an infinite time horizon $T= + \infty$. In Theorem \ref{th:sufficient}, we will provide a sufficient condition under which the above non-spurious trajectory property holds for a general objective function with a damping sinusoidal time-varying perturbation.
}

\begin{figure*}
      \centering
      \subcaptionbox{Graph of a time-varying optimization $\inf_{x\in \mathbb{R}} f(x,t)$ showing that the final state of the trajectory belongs to the region of attraction of the global minimum.\label{fig:naive_spurious_1}}
        {  \includegraphics[width=0.42\linewidth]{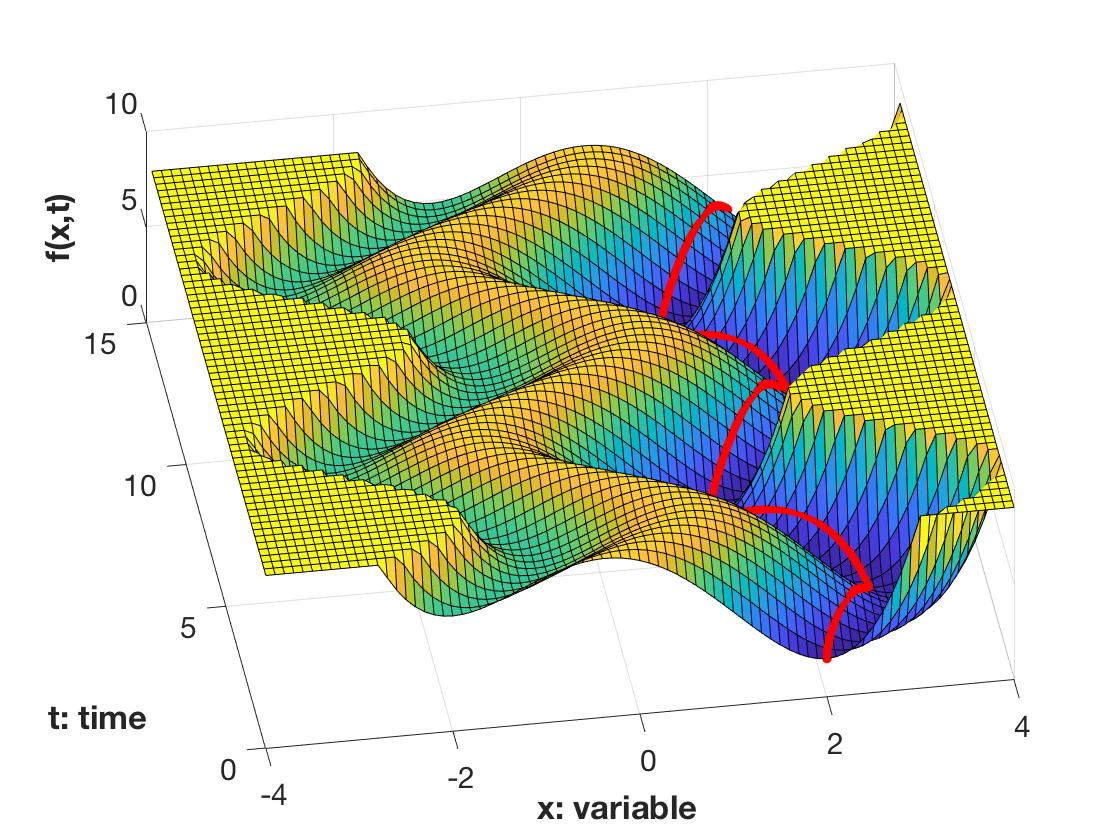}}\hspace{1mm}
      \subcaptionbox{Graph of the same time-varying optimization $\inf_{x\in \mathbb{R}} f(x,t)$ from above showing that the trajectory can never stay in a neighborhood of the global minimum of arbitrarily small size.\label{fig:naive_spurious_2}}
        { \includegraphics[width=0.43\linewidth]{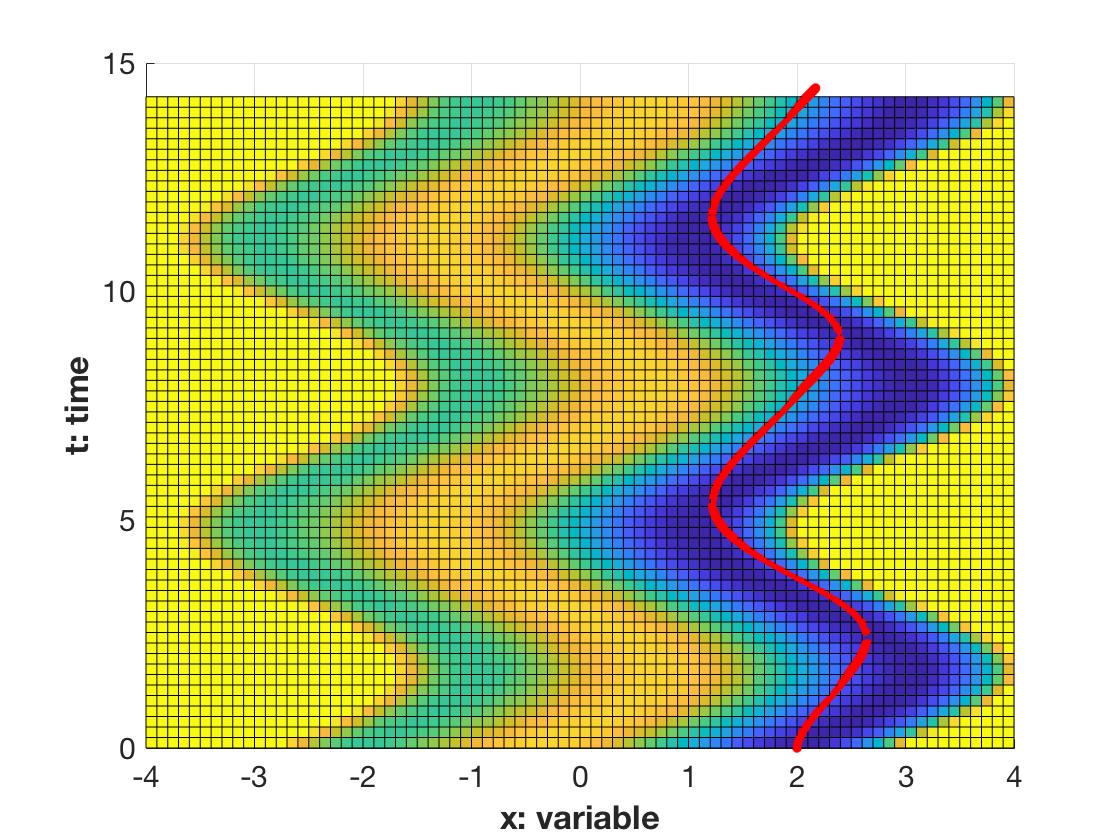}}
      \caption{Example of a time-varying optimization.}
    \end{figure*}
{It may be speculated that  a spurious local trajectory could have been simply defined as a trajectory that does not converge towards a global solution. To understand why the latter definition is not meaningful, notice that both discrete and continuous local trajectories are defined with respect to the regularized problem~\eqref{obj_reg}, as opposed to~\eqref{eq:problem}. The regularization term acts as an \textit{inertia} in the continuous local trajectory, forcing it to ``lag behind'' the global solution when it changes rapidly over time. Therefore, under this alternative definition, all trajectories would be considered spurious. This would be true even for the trajectory initialized at the global minimum. See Figures \ref{fig:naive_spurious_1} and \ref{fig:naive_spurious_2} for an illustration of this phenomenon.}

{
The notion introduced in Definition \ref{def:spurious}, while it deals with continuous local trajectories, naturally has implications for discrete local trajectories. With sufficiently small time steps, the discrete trajectory will eventually converge to the region of attraction of a global solution if the corresponding continuous trajectory is not spurious.}


\section{{Conditions for the Absence of Spurious Local Trajectories}}
\label{sec:Conditions for the Absence of Spurious Local Trajectories}

{In this section, we analyze the role of data variation on the behavior of the solution trajectories. Observe that without data variation, strict spurious local minima cannot not be escaped. This is a consequence of classical results on the local stability of time-invariant ODEs (see for instance \cite[Corollary 10]{tanabe1974algorithm}). In contrast, we show that data variation can enable escaping spurious local solutions over time. In particular, we prove that even a simple periodic variation in the data can induce continuous local trajectories to escape non-global minima and eventually track the global minima.} 

{To better illustrate the main idea, we start with a class of uni-dimensional time-varying problems, and provide sufficient conditions for the absence of spurious local trajectories. Then, we extend our results to a general class of multi-dimensional problems. Consider the function}
\begin{equation}
\inf_{x\in \mathbb{R}} f(x,t) := g(x-\beta \sin(t))
\label{eq:eg}
\end{equation}
where $g: \mathbb{R} \longrightarrow \mathbb{R}$ is continuously twice differentiable and $\beta >0$ models the variation of the data over time. Only the right-hand side varies over time, and therefore, this problem fits well in our introduced framework. We assume that $g(\cdot)$ admits only three stationary points $g'(y_1) = g'(y_2) = g'(y_3)$ with $y_1<y_2<y_3$. We assume also that $y_1$ and $y_3$ are local minima such that $g(y_1)>g(y_3)$, while $y_2$ is a local maximum. Finally, we assume that $g$ is coercive (its limit at $\pm \infty$ is $+\infty$). Thus, its global infimum is reached in $y_3$.

The motivation behind studying this class of functions $f(\cdot)$ is as follows. Since $g(y)$ has a global minimum as well as a spurious solution, when it is minimized by a gradient descent algorithm initialized at the spurious solution, it will become stuck there. This means that using gradient descent for such function is inefficient. However,  one can oscillate the function to arrive at the time-varying function $f(x,t)$ and then study it in the context of online optimization. The following result identifies sufficient conditions for the absence of spurious local trajectories, which implies  that if $\alpha$ and $\beta$ are selected appropriately, gradient descent will always find the global solution.

\begin{proposition}\label{prop:uni}
	\label{prop:sufficient}
	If $\alpha,\beta>0$ are such that
	\begin{enumerate}
		\item $\alpha \beta \geqslant C:= \max_{y_1 \leqslant y \leqslant y_3} g'(y)$,
		\item $\exists m_1,m_2 \in \mathbb{R}: ~~ m_1< y_1 < m_2 ~~~\text{and}~~~ g'(m_1) = g'(m_2) = -  \alpha \beta$,
		\item $ -C/\alpha (t_2-t_1) -\beta (\sin(t_2) - \sin(t_1)) + m_1 \geqslant m_2 $ where $0 < t_1 \leqslant t_2$ satisfy $\cos(t_1) = \cos(t_2) = -C/(\alpha \beta)$,
	\end{enumerate}
	then the time-varying problem \eqref{eq:eg} has no spurious local trajectories.
\end{proposition}
\begin{proof}
    A continuous local trajectory $x : [0,2\pi] \longrightarrow \mathbb{R}$ satisfies
\begin{equation}
x(0)\leqslant y_3,\qquad \dot{x} = -\frac{1}{\alpha}\nabla_x f(x,t),
\end{equation}
which, after the change of variable $y := x - \beta \sin(t)$, reads
\begin{equation}
y(0)\leqslant y_3,\qquad \dot{y} = -\frac{1}{\alpha}g'(y) - \beta \cos (t).
\end{equation}
We first show by contradiction that there exists $t \in [0,2\pi]$ such that $y(t) \geqslant m_2$. Assume that $y(t)<m_2$ for all $t\in [0,2\pi]$. Then, for all $t\in [0,2\pi]$, it holds that 
\begin{equation}
\dot{y} = -\frac{1}{\alpha}g'(y) - \beta \cos(t) \geqslant -\frac{C}{\alpha} - \beta \cos(t).
\end{equation}
Thus, we have
\begin{equation}
y(t_2) \geqslant -\frac{C}{\alpha}(t_2-t_1) - \beta(
\sin(t_2) - \sin(t_1)) + y(t_1).
\label{eq:t1t2}
\end{equation}
We next show by contradiction that $y(t_1) \geqslant m_1$. Assume that $y(t_1)<m_1$. Thus $y(t_1)<m_1<y_1\leqslant y(0)$. Let $t_3$ denote the maximal element of the compact set $[0,t_1] \cap y^{-1}(m_1)$, where $y^{-1}(b) := \{ a \in \mathbb{R} ~|~ y(a) = b\}$. Thus $y(t) \leqslant y(t_3)$ for all $t \in [t_3,t_1]$. As a result, $y'(t_3) \leqslant 0$. Together with $y'(t_3) = - 1/\alpha g'(m_1) - \beta \cos(t_3) = \beta (1-\cos(t_3))$, this implies that $t_3 = 0$ or $t_3 = 2\pi$. This is in contradiction with $0<t_3<t_1 < \pi$.

Now that we have proven that $y(t_1) \geqslant m_1$, equation \eqref{eq:t1t2} implies that $y(t_2) \geqslant m_2$. This is a contradiction. Therefore there exists $t \in [0,2\pi]$ such that $y(t) \geqslant m_2$. Using the same argument as in the previous paragraph, we obtain $y(2\pi) \geqslant m_2$. As a result, $x(2\pi) = y(2\pi) - \beta \sin(2\pi) \geqslant m_2$ as well. {Finally, using standard arguments in Lyapunov theory, there exists $\bar{T} < T$ such that $x(t)$ belongs to the region of attraction of $y_3$ for  all $t\in[\bar T,T]$.}
\end{proof}

We highlight the implications of the above proposition through a numerical example.

\begin{example}\label{exp1}
	Consider the objective function $f(x,t) := g(x- \beta \sin(t))$
	where 
	\begin{equation}
	g(y) := 1/4y^4 + 1/8 y^3 - 2y^2  - 3/2y + 8.
	\end{equation}
	The time-varying objective $f(x,t)$ has the following stationary points: it admits a spurious local minimum at $-2+\beta \sin(t)$, a local maximum at $-3/8 + \beta \sin(t)$, and a global minimum at $2+\beta \sin(t)$. The three sufficient conditions of Proposition \ref{prop:sufficient} can be brought to bear on this example. They yield three inequalities, as shown in Figure \ref{fig:ineq}, whose feasible region is represented in Figure \ref{fig:graph}. Taking a point in that feasible region, we confirm numerically in Figure~\ref{fig:non-spurious} that a trajectory initialized at a local minimum of $f(\cdot,0)$ winds up in the region of attraction of the global solution to $f(\cdot,T)$ at the final time $T=2\pi$. In contrast, taking a point outside the feasible region, we observe in Figure~\ref{fig:spurious} that a trajectory initialized at a local minimum of $f(\cdot,0)$ does not end up in the region of attraction of the global solution to $f(\cdot,T)$.\footnote{In order to increase visibility, a maximal threshold is used on the objective function $f(x,t)$ in Figure \ref{fig:non-spurious} and Figure \ref{fig:spurious} (hence the flat parts). For the same reason, a non-linear scaling is used. Precisely, $(x,t) \longrightarrow f(x+(\beta-1)\sin(t),t)$ and $t \longrightarrow x(t) - (\beta-1) \sin(t)$ are represented in the figures. This explains why $x(t)$ appears to decrease for small $0 \leqslant t \leqslant 2\pi$ in Figure \ref{fig:non-spurious}.}

	We make a few remarks regarding Figure \ref{fig:ineq}. Note that $k_1$ and $k_2$ are integers in $\{0,1,2\}$ such that $k_1$ minimizes the line it appears in, and $k_2$ minimizes the line it appears in while not being equal to $k_1$. These numbers come from Vi\`ete's solution to a cubic equation \cite{nickalls2006}. Furthermore, the second inequality corresponds to minus the discriminant of a fourth-order polynomial.
\end{example}

{Next, we will extend the aforementioned result to a general class of multi-dimensional optimization problems. The goal is to show that certain non-global local solutions of an arbitrary time-invariant function $g(x)$ that cannot be escaped using deterministic local search methods can indeed be escaped via the conversion of the problem to a time-varying function $f(x,t)$ for which there is no spurious trajectory. Consider the time-varying optimization problem
    \begin{equation}
    \label{eq:multi}
        	\inf_{x\in \mathbb{R}^n} f(x,t) := \inf_{x\in \mathbb{R}^n}g(x-\beta e^{-\lambda t} \sin(\omega t) u)
    \end{equation}
    where $g: \mathbb{R}^n \longrightarrow \mathbb{R}$ is continuously twice differentiable, coercive (its limit as $\|y\| \rightarrow +\infty$ is $+\infty$). The amplitude $\beta>0$ and the pulsation $\omega>0$ model the sinusoidal variation of data over time with a damping factor of $\lambda>0$. The variation occurs along a direction $u \in \mathbb{R}^n$ of norm 1. Let $\{y_i\}_{i\in \mathcal{I}}$ denote the set of spurious local minima of $g(x)$. Moreover, let $B(a,r)$ (respectively $S(a,r)$) denote the Euclidian ball (respectively sphere) in $\mathbb{R}^n$ centered at $a$ and of radius $r$. Given a fixed $R>0$, we define the following constants
	\begin{align}
	     C_1 &:= \max\limits_{\scriptsize y\in \bigcup\limits_{i\in \mathcal{I}} B(y_i,R)} \|\nabla g(y)\|,\nonumber\\[.2cm]
	     C_2 &:= \min\limits_{\tiny \begin{array}{c} d \in S(0,1) \\ i\in \mathcal{I} \end{array}} \langle \nabla g(y_i-Rd) , d \rangle.
	\end{align}
	These constants enable us to control fluctuations of $g(x)$ in the vicinity of its local minima. For the sake of clarity, we assume that $g(x)$ has no saddle points and local maxima outside of $\cup_{ i\in \mathcal{I}} B(y_i,R)$ (for more on this, see Remark \ref{rem:saddle}). 
	Notice that $C_1\geqslant C_2$ due to the Cauchy-Schwarz inequality. Theorem \ref{th:sufficient} below shows that if $C_1$ is not too large, then one can escape spurious local minima, and if $C_2$ is not too small, then one will never return to the vicinity of any spurious local minima after some time. \begin{theorem}
	\label{th:sufficient}
    If 
    $ 2 \alpha \omega (\beta e^{-\lambda \pi/(2\omega)} - R )/\pi  > C_1 $
    ~and~ $\alpha \beta e^{-\lambda R\alpha/(C_1 +\alpha \beta \omega)} \sqrt{\lambda^2+\omega^2} <  C_2$, then 
    the time-varying optimization \eqref{eq:multi} has no spurious trajectories.
	\end{theorem}
	\begin{proof}
	First, we show that the spurious local minimum is initially escaped. A continuous local trajectory $x(t)$ satisfies
\begin{equation}
x(0)\in \{y_i\}_{i\in \mathcal{I}},\qquad x'(t) = -\frac{1}{\alpha}\nabla_x f(x(t),t),
\end{equation}
which, after the change of variables $y(t) := x(t) - \beta e^{-\lambda t} \sin(\omega t) u$, reads
\begin{align}
y'(t) &= -\nabla g(y(t))/\alpha -\beta e^{-\lambda t}[-\lambda  \sin(\omega t) +  \omega \cos(\omega t)]u,\nonumber\\
y(0)&\in \{y_i\}_{i\in \mathcal{I}},
\end{align}
We first show by contradiction that there exists some time $t\in[0,T]$ such that $\|y(t)- y(0)\| > R > 0$. Assume that $\|y(t)- y(0)\| \leqslant R$ for all $t\geqslant 0$. Then, for all $t\geqslant 0$, it holds that
\begin{align}
         & \langle y'(t) , u \rangle \nonumber\\
         & = \langle -\nabla g(y(t))/\alpha -\beta e^{-\lambda t}[-\lambda  \sin(\omega t) +  \omega \cos(\omega t)] u , u \rangle \nonumber\\
         & = -\langle \nabla g(y(t)) , u \rangle/\alpha  -\beta e^{-\lambda t}[-\lambda  \sin(\omega t) +  \omega \cos(\omega t)] \langle  u , u \rangle\nonumber\\
         & \leqslant \|\nabla g(y(t))\| / \alpha -\beta e^{-\lambda t}[-\lambda  \sin(\omega t) +  \omega \cos(\omega t)] \nonumber \\
         & \leqslant \{C_1 -\alpha \beta e^{-\lambda t}[-\lambda  \sin(\omega t) +  \omega \cos(\omega t)] \}/ \alpha,
\end{align}
from which we deduce that
\begin{align}
     \langle y(t) - y(0) , u \rangle  &=  \left\langle \int_0^t y'(s) ds , u \right\rangle
      =  \int_0^t \langle y'(s) , u \rangle ds\nonumber\\
      &\leqslant [ C_1t - \alpha \beta e^{-\lambda t} \sin(\omega t) ]  / \alpha.\label{eq:int}
\end{align}
Our assumption that $2 \alpha \omega (\beta e^{-\lambda \pi/(2\omega)} - R )/\pi  > C_1 $ implies that the upper bound in \eqref{eq:int} is negative when $t=\pi/(2\omega)$. Using the Cauchy-Schwarz inequality, we then obtain
\begin{align}
\|y(\pi/(2\omega))-y(0)\| &\geqslant |\langle y(\pi/(2\omega)) - y(0) , u \rangle| \nonumber\\
\nonumber&\geqslant [ \alpha \beta e^{-\lambda \pi/(2\omega)} - C_1 \pi/(2\omega)  ]  / \alpha > R.
\end{align}
This yields a contradiction. We conclude that there exists $ t_1 \geqslant 0$ such that $\|y(t_1)-y(0)\| > R$. Observe that 
\begin{align}
     \|y(t_1) - y(0)\| & = \left\| \int_{0}^{t_1} \nabla g(y(t)) dt - \beta e^{-\lambda t_1} \sin(\omega t_1)u \right\|\nonumber\\
     & = \int_{0}^{t_1} \left\|\nabla g(y(t)) dt\right\| 
     + \beta e^{-\lambda t_1} \sin(\omega t_1) \nonumber\\
     & \leqslant C_1t_1/\alpha + \beta e^{-\lambda t_1}  \sin(\omega t_1)\nonumber\\
     & \leqslant (C_1/\alpha + \beta \omega)t_1.
\end{align}
As a result, $t_1 > R\alpha / (C_1 +\alpha \beta \omega)$. We have thus identified a minimum time taken by the trajectory to exit the ball of radius $R$ centered at $y(0)$. 
Second, we show that, after some time, the continuous trajectory never returns to the vicinity of any spurious local minimum. To reason by contradiction, assume that there exist $i \in \mathcal{I}$ and $t_1 < t_3$ such that $\|y(t_3)-y_i\| < R$. Since the trajectory is continuous, there exists $t_2\in(t_1,t_3)$ such that $\|y(t_2)-y_i\| = R$, that is to say, there exists $d \in \mathbb{R}^n$ such that $\|d\|=1$ and $y(t_2) = y_i + R d$. Take $t_2$ to be the largest such instance in the interval $(t_1,t_3)$. We then have
\begin{align}
         &\langle y'(t_2) , d \rangle\nonumber\\
         & = \langle -\nabla g(y(t_2))/\alpha -\beta e^{-\lambda t_2}[-\lambda  \sin(\omega t_2) +  \omega \cos(\omega t_2)]u , d \rangle \nonumber\\
         & = \langle \nabla g(y_i+Rd) , -d \rangle / \alpha  -\beta e^{-\lambda t_2}[-\lambda  \sin(\omega t_2) +  \omega \cos(\omega t_2)] \langle  u , d \rangle \nonumber\\
         & \geqslant C_2 / \alpha - \beta e^{-\lambda t_2}[-\lambda  \sin(\omega t_2) +  \omega \cos(\omega t_2)] \langle  u , d \rangle \nonumber\\
         & = \left\{C_2  - \alpha \beta e^{-\lambda t_2}\sqrt{\lambda^2+\omega^2} \cos(\omega t_2 + \arctan(\lambda/\omega))]\right\} / \alpha\nonumber\\
         & \geqslant (C_2  - \alpha \beta e^{-\lambda t_2} \sqrt{\lambda^2+\omega^2} ) / \alpha \nonumber\\
         & \geqslant (C_2  - \alpha \beta e^{-\lambda R\alpha/(C_1 +\alpha \beta \omega)} \sqrt{\lambda^2+\omega^2} ) / \alpha
\end{align}
where in the last inequality we used the fact that $R\alpha/(C_1 +\alpha \beta \omega) \leqslant t_1 < t_2$. The Taylor expansion for $t>t_2$ in a neighborhood of $t_2$ reads
\begin{equation}
\label{eq:taylor}
      y(t) - y(t_2) = y'(t_2)(t-t_2) + o( t - t_2),
\end{equation}
from which we deduce that
\begin{align}
\label{eq:taylorbis}
&\left\langle \frac{y(t) - y(t_2)}{t-t_2} , d \right\rangle \nonumber\\
&= \langle y'(t_2) , d \rangle + o(1)\nonumber\\
&> (C_2  - \alpha \beta e^{-\lambda R\alpha/(C_1 +\alpha \beta \omega)} \sqrt{\lambda^2+\omega^2} ) / (2 \alpha) > 0
\end{align}
where we used
$\alpha \beta e^{-\lambda R\alpha/(C_1 +\alpha \beta \omega)} \sqrt{\lambda^2+\omega^2} < C_2$. Hence 
\begin{align}
     \| y(t) - y_i \|  
      \geqslant  \langle y(t)-y_i , d \rangle 
      &=  \langle y(t)-y(t_2)+y(t_2)-y_i , d \rangle \nonumber\\
      &=  \langle y(t)-y(t_2),d \rangle + \langle Rd , d \rangle\nonumber\\
      &> R.
\end{align}
Recall that $\|y(t_3)-y(0)\|\leqslant R$. By continuity of the trajectory, there exists $t\in(t_2,t_3]$ such that $\| y(t) - y_i \| = R$, which contradicts the maximality of $t_2$. Hence, for all $t \geqslant t_1 $ and $i \in \mathcal{I}$, we have that $\|y(t) - y_i\| \geqslant R$. \\
Third, we show that $x(t_1)=y(t_1)+\beta e^{-\lambda t_1} \sin(wt_1)u$ is in the region of attraction of a global minimum of the function $f(x,t_1)$. Now, we freeze the time at $t_1$. Consider the set $D=\{x\in \mathbb{R}^n: f(x,t_1)\leq f(x(t_1),t_1)\}$ and choose $D_1$ as the connected component of $D$ which contains the point $x(t_1)$. Because $f(x,t_1)$ is coercive, $D_1$ is a compact set. In addition, $D_1$ is a positively invariant set with respect to the gradient flow system
\begin{equation}\label{eq: gradient flow at t1}
    \dot{\tilde{x}}(s) = -\nabla_{\Tilde{x}} f(\Tilde{x}(s),t_1)
\end{equation}
 for the fixed time $t_1$ because the gradient flow system will not increase the function value. Denote $f^\ast(t_1)$ as the global minimum value of $f(\Tilde{x},t_1)$ and take $V(\Tilde{x})=f(\Tilde{x},t_1)-f^\ast(t_1)$. Then, $V(\Tilde{x})$ is a Lyapunov function for \eqref{eq: gradient flow at t1} such that  $\dot V(\Tilde{x})=- \|\nabla_{\Tilde{x}} f(\Tilde{x},t_1)\|^2\leq 0$ in $D_1$. Let $E$ be the points in $D_1$ such that $\nabla_{\Tilde{x}} f(\Tilde{x},t_1)=0$. Since $g(x)$ has no saddle points and local maxima outside of $\cup_{ i\in \mathcal{I}}B(y_i,R)$, then $f(\cdot,t_1)$ has no saddle points and local maxima outside of $\cup_{ i\in \mathcal{I}}B(y_i+\beta e^{-\lambda t_1} \sin{(wt_1)}u,R)$. Thus,
 the set $E$ only contains the global minima of $f(\Tilde{x},t_1)$. Furthermore, the set $E$ is also an invariant set with respect to \eqref{eq: gradient flow at t1}. Then, by LaSalle's theorem in  \cite[Theorem 4.4]{khalil2002nonlinear}, the solution of \eqref{eq: gradient flow at t1} starting at $x(t_1)$ converges to the global minimum as $s \rightarrow \infty$. This implies that $x(t_1)$ is in the region of attraction of a global minimum of the function $f(x,t_1)$. Finally, we show that the trajectory remains in the region of attraction of the set of global minima after some time. This follows immediately from the assumption that $g(x)$ has no saddle points and local maxima outside of $\cup_{ i\in \mathcal{I}} B(y_i,R)$ and the fact that the trajectory will never returns to the vicinity of any spurious local minimum, that is,  $\cup_{ i\in \mathcal{I}} B(y_i,R)$.
	\end{proof}}
	{
	Observe that a necessary condition for the absence of spurious trajectories readily follows from the proof of Theorem \ref{th:sufficient}, namely that $\alpha \beta \sqrt{\omega^2+\lambda^2} \geqslant -C_2$. Indeed, if $\alpha \beta \sqrt{\omega^2+\lambda^2} < -C_2$, then the spurious local minima cannot be escaped, using the same argument as in \eqref{eq:taylor} and \eqref{eq:taylorbis}.}
{\begin{remark}
\label{rem:saddle}
    In Theorem \ref{th:sufficient}, we assume that there are no saddle points or maxima outside of a certain region containing the local minima (i.e. $\cup_{ i\in \mathcal{I}} B(y_i,R)$). We do so in order to focus on the main contribution of this work, which is that time variation can lead to the absence of spurious local trajectories. Without this assumption, a significant part of the proof would deal with escaping saddle points, a subject which has already been treated in various papers \cite{lee2016,panageas2017,davis2019,chinge2017,criscitiello2019}. If the variation of the data occurs along a direction $u$ chosen randomly, then it may be argued that the trajectory would escape saddle points with probability 1, using the stable manifold theorem~\cite{chicone2006ordinary} as in \cite{lee2016,panageas2017,davis2019,chinge2017,criscitiello2019}. Theorem \ref{th:sufficient} would then hold almost surely.
    \begin{remark}
    Theorem 3 offers the first result in the literature about when spurious minima of a time-invariant function can be escaped via a time-varying deterministic local search method. The existing results are focused on stochastic gradient descent that offers a weaker result in a probabilistic sense \cite{pmlrv80kleinberg18a}. This theorem can be used to define the notion of escapable local minima through the parameters $C_1$ and $C_2$, and indeed if $C_1$ is small enough and $C_2$ is large enough, the spurious local minima can always be escaped based on the results of this theorem. 
   \end{remark}
\end{remark}
Although Theorem \ref{th:sufficient} is focused on a certain class of time-varying functions, similar results can be obtained for other classes of functions. The time-varying problem~\eqref{eq:problem} is devoid of spurious local trajectories if one can show that all solutions of ~\eqref{eq:ode} with the initial point at any local solutions at $t=0$ are contractive and the converging trajectory is inside the region of attraction of the global minimum trajectory  of \eqref{eq:problem} after some finite time. This can be studied via the contraction analysis of nonlinear systems \cite{lohmiller1998contraction, lohmiller2000nonlinear, lu2016contraction}. 
}
    
\section{{Fundamental Properties of ODE}}
\label{sec:theo_found}
		In this section, we provide the formal versions of Theorems~\ref{thm:exist_unique} and~\ref{thm:conv} together with their proofs. 
	{We refer to the optimization problem \eqref{obj_reg} as $\mathrm{OPT}(k,\Delta t, x_{k-1})$.}
	Let the Jacobian of the constraint set be defined as
	\begin{align}
	\mathcal{J}(x) = \begin{bmatrix}
	\nabla_x h_1(x)^\top\\
	\nabla_x h_2(x)^\top\\
	\vdots\\
	\nabla_x h_r(x)^\top
	\end{bmatrix}
	\end{align}
	
	\begin{definition}\label{def1}
		\begin{sloppypar}
			{Given a feasible initial point $x_0$, we say that the tuple $\left(x_0,\Delta t,\{x_k^{\Delta t}\}_{k=0}^{\infty}\right)$ is an \textbf{admissible KKT (AKKT) tuple} if $x_0^{\Delta t} = x_0$ and for every $k\in\{0,1,...\}$, $x^{\Delta t}_k$ is a feasible solution of $\mathrm{OPT}(k,\Delta t, x^{\Delta t}_{k-1})$, it satisfies the KKT conditions, and $\mathcal{J}(x^{\Delta t}_k)$ is non-singular.}
		\end{sloppypar}
	\end{definition}
	
	\begin{assumption}\label{assum1}
		{There exists $\overline{t}>0$ such that any $0<\Delta t\leq \overline{t}$ is endowed with at least one AKKT tuple $\left(x_0,\Delta t,\{x_k^{\Delta t}\}_{k=0}^{\infty}\right)$. Furthermore, for any AKKT tuple $\left(x_0,\Delta t,\{x_k^{\Delta t}\}_{k=0}^{\infty}\right)$, the sequence $\left\{x_0, \{x_k^{\Delta t}\}_{k=0}^{\infty}\right\}$ is uniformly bounded.}
	\end{assumption}
	{Roughly speaking, Assumption~\ref{assum1} implies that, for sufficiently small time steps, the regularized problem remains feasible with non-degenerate and bounded solutions.}
	
	
	{According to  Definition~\ref{def1}, the Jacobian matrix $\mathcal{J}(x_k^{\Delta t})$ is non-singular for every {$k$} and every AKKT tuple $\left(x_0,\Delta t,\{x_k^{\Delta t}\}_{k=1}^{\infty}\right)$.
	In this work, we impose a slightly stronger condition on the singular values of $\mathcal{J}(x_k^{\Delta t})$.}
	
	\begin{assumption}\label{assum2}
		There exists a universal constant $c>0$ such that $\sigma_{\min}(\mathcal{J}(x_k^{\Delta t}))\geq c$ for every $k$ and every AKKT tuple $\left(x_0,\Delta t,\{x_k^{\Delta t}\}_{k=0}^{\infty}\right)$.
	\end{assumption}
	
	{Similar to Assumption~\ref{assum22}, this assumption requires the constraints to be non-degenerate.} Now, we are ready to present our main theorem.
	
	\begin{theorem}\label{thm:ode}
		{Consider the ODE~\eqref{eq:ode}
		with the condition $x(0) = x_0$, where $x_0$ is a local solution to the time-varying optimization~\eqref{eq:problem} at $t = 0$. The following statements hold:}
		\begin{itemize}
			\item[1.] (Existence and uniqueness)~\eqref{eq:problem} has a continuously differentiable and unique {solution $x:[0,T]\to\mathbb{R}^n$.}
			\item[2.] {(Convergence) Any AKKT tuple $\left(x_0,\Delta t,\{x_k^{\Delta t}\}_{k=0}^{\lceil T/\Delta t\rceil}\right)$ satisfies
			\begin{align}\label{converge1}
			\lim\limits_{\Delta t \to 0^+}\sup_{0\leq k\leq {\lceil T/\Delta t\rceil}}\|x_k^{\Delta t}-x(k\Delta t)\| = 0,
			\end{align}}
		\end{itemize}
	\end{theorem}

We will regularly refer to the following lemma in our subsequent analysis.
\begin{lemma}[Lipschitz property on a ball]\label{lem:fun}
	Given a continuously differentiable function $p(x):\mathbb{R}^n\to\mathbb{R}^m$, we have 
	\begin{align*}
	\|p(x)-p(y)\|\leq L(\epsilon)\|x-y\|\qquad \text{for every } x,y\in\mathcal{B}(\epsilon)
	\end{align*}
	where $L(\epsilon)$ is a universal constant independent of $x$ and $y$, and $\mathcal{B}(\epsilon)$ is the Euclidean ball centered at zero with radius $\epsilon$.
\end{lemma}

\begin{proof}
{The proof is straightforward and omitted for brevity.}
\end{proof}

\subsection{Proof of Existence and Uniqueness.} Next, we show the existence and uniqueness of the solution to the proposed ODE. {Without loss of generality, we assume that $t_k-t_{k-1} = \Delta t$ for every $k = 1,\dots, {\lceil T/\Delta t\rceil}$.} Furthermore, to simplify the notation, we may use the same symbols to refer to different universal constants throughout the proofs. The next three lemmas will be useful in proving the existence of a solution~\eqref{eq:ode}.

\begin{lemma}\label{lem:cont2}
	There exist constants $\bar{t}$ and $c>0$ such that for every AKKT tuple {$\left(x_0,\Delta t,\{x_k^{\Delta t}\}_{k=0}^{\lceil T/\Delta t\rceil}\right)$} with $\Delta t\leq \bar{t}$, we have $\|x_k^{\Delta t}-x_{k-1}^{\Delta t}\|\leq c\Delta t$ for  $k = 1,\dots,\infty$.
\end{lemma}
{\begin{proof}
The proof is provided in the appendix.
\end{proof}}
\begin{lemma}\label{lem:cont}
	Given an initial feasible point $x_0$, there exist
	\begin{itemize}
		\item[1.] $\{s_n\}_{n=1}^{\infty}$ with $\lim\limits_{n\to\infty}s_n = 0$ such that each $s_n$ is endowed with an AKKT tuple $(x_0,s_n, \{x_k^{s_n}\}_{k=0}^{\infty})$, and
		\item[2.] a continuously differentiable and uniformly bounded function $\bar{x} : [0,T]\to\mathbb{R}^n$ that satisfies $\bar x(0) = x_0$,
	\end{itemize}
	with the following properties:
	\begin{subequations}
	\begin{align}
	& \lim\limits_{n\to\infty}\sup_{1\leq k\leq \frac{T}{s_n}}\left\|x_k^{s_n}-{\bar{x}}(k s_n)\right\| = 0,\\
	& \lim\limits_{n\to\infty}\sup_{1\leq k\leq \frac{T}{s_n}}\left\|\frac{x_k^{s_n}-x_{k-1}^{s_n}}{s_n}-\dot{\bar{x}}(ks_n)\right\| = 0.\label{limsupx}
	\end{align}
	\end{subequations}
	Moreover, there exists a universal constant $c>0$ such that $\sigma_{\min}(\mathcal{J}(\bar{x}(t)))\geq c$ for every $t\in[0,T]$.
\end{lemma}

{\begin{proof}
The proof is provided in the appendix.
\end{proof}}

\begin{lemma}\label{lem:equiv}
	Consider two continuous functions $g_1:[0,T]\to\mathbb{R}^n$ and $g_2:[0,T]\to\mathbb{R}^n$. We have $g_1 = g_2$ if and only if
	\begin{align}
	\lim\limits_{\Delta t\to 0^+}\sup_{0\leq k\leq \lceil\frac{T}{\Delta t}\rceil}\|g_1(k\Delta t) - g_2(k\Delta t)\| = 0
	\end{align}
\end{lemma}
{\begin{proof}
The proof is straightforward and can be found in standard references, e.g.,~\cite{rudin2006real}.
\end{proof}}

We now provide the proof for the existence and uniqueness of the solution for~\eqref{eq:ode}.

\vspace{2mm}
{\it Proof of existence and uniqueness:} Consider the sequence $\{s_n\}_{n=1}^{\infty}$ and its corresponding AKKT tuple $\left\{(x_0,s_n, \{x_k^{s_n}\}_{k=0}^{T/s_n})\right\}_{n=1}^\infty$ that is introduced in Lemma~\ref{lem:cont}. Due to Assumption~\ref{assum2}, the linear independence constraint qualification (LICQ) holds at $x_k^{s_n}$ for  $k=0,\dots,T/s_n$ and $n=1,\dots,\infty$. Therefore, for every $n$, there exists a sequence of Lagrangian vectors $\{\mu_k^{s_n}\}_{k=0}^{T/s_n}$ such that $\left(\{x_k^{s_n}\}_{k=0}^{T/s_n},\{\mu_k^{s_n}\}_{k=0}^{T/s_n}\right)$ satisfies the KKT conditions:
\begin{align}
&\nabla_x f_k(x_k^{s_n})+\mathcal{J}(x_k^{s_n})^\top\mu_k^{s_n}+\frac{\alpha}{s_n}(x_k^{s_n}-x_{k-1}^{s_n}) = 0\tag{Stationarity}\\
& h_i(x_k^{s_n}) = d_{i,k}\tag{feasibility}
\end{align}
for  $k = 1,\dots,T/s_n$, where $f_k(x_k^{s_n}) = f(x_k^{s_n}, ks_n)$ and $d_{i,k} = d_i(ks_n)$. {The feasibility condition implies that for every $i$, we have}
\begin{align}
&\frac{1}{s_n}\left(h_i(x_k^{s_n})-h_i(x_{k-1}^{s_n})\right) = \frac{d_{i,k}-d_{i,k-1}}{s_n}\nonumber\\
\implies & \nabla h_i(\tilde{x}^{s_n}_{i,k})^\top\left(\frac{x^{s_n}_k-x_{k-1}^{s_n}}{s_n}\right) =  \frac{d_{i,k}-d_{i,k-1}}{s_n}
\end{align}
for some $\tilde{x}_{i,k}^{s_n} = (1-\alpha_i)x_k^{s_n}+\alpha_i x_{k-1}^{s_n}$ with $\alpha_i\in[0,1]$, where the last implication is due to the differentiability of $h_i(x)$ and the Mean Value Theorem. For simplicity and with a slight abuse of notation, define 
\begin{align}
\mathcal{J}(\{\tilde{x}_{i,k}^{s_n}\}_{i=1}^r) = \begin{bmatrix}
\nabla h_1(\tilde{x}_{1;k}^{s_n})^\top\\
\vdots\\
\nabla h_r(\tilde{x}_{r;k}^{s_n})^\top
\end{bmatrix},\quad d_k = \begin{bmatrix}
d_{1,k}\\
\vdots\\
d_{r,k}
\end{bmatrix}
\end{align}
This implies that
\begin{align}\label{jacob}
\mathcal{J}(\{\tilde{x}_{i,k}^{s_n}\}_{i=1}^r) \left(\frac{x_k^{s_n}-x_{k-1}^{s_n}}{s_n}\right) = \frac{d_{k}-d_{k-1}}{s_n}
\end{align}
Combining this equality with the stationarity condition leads~to
\begin{align}
\mathcal{J}(\{\tilde{x}_{i,k}^{s_n}\}_{i=1}^r)\nabla_x f_k(x_k^{s_n})+\mathcal{J}(\{\tilde{x}_{i,k}^{s_n}\}_{i=1}^r)\mathcal{J}(x_k^{s_n})^\top\lambda_k^{s_n}
+\alpha \left(\frac{d_{k}-d_{k-1}}{s_n}\right) = 0
\end{align}
Now, note that, due to Assumption~\ref{assum2}, $\sigma_{\min}(\mathcal{J}(x_k^{s_n}))\geq c$ for some universal constant $c>0$. Therefore, for every $y$ sufficiently close to $x_k^{s_n}$, $\mathcal{J}(y)$ remains full-row rank. {Together with the definition of $\{\tilde{x}_{i,k}^{s_n}\}_{i=1}^r$ and Lemma~\ref{l_app} in the appendix}, this implies that $\mathcal{J}(\{\tilde{x}_{i,k}^{s_n}\}_{i=1}^r)\mathcal{J}(x_k)^\top$ is invertible for sufficiently small $\Delta t$. Therefore,
\begin{align}
\lambda_k^{s_n} =& - \left(\mathcal{J}(\{\tilde{x}_{i,k}^{s_n}\}_{i=1}^r)\mathcal{J}(x_k^{s_n})^\top\right)^{-1}\times\left(\mathcal{J}(\{\tilde{x}_{i,k}^{s_n}\}_{i=1}^r)\nabla_x f_k(x_k^{s_n})+\alpha \left(\frac{d_{i;k}-d_{i;k-1}}{s_n}\right)\right).\nonumber
\end{align}
Substituting this into the stationarity condition and performing the necessary simplifications lead to
\begin{align}\label{discrete}
\frac{x_k^{s_n}\!-\!x_{k-1}^{s_n}}{s_n} \!=& -\frac{1}{\alpha}\Bigg(I\!-\!\mathcal{J}(x_k^{s_n})^\top\times\left(\mathcal{J}(\{\tilde{x}_{i,k}^{s_n}\}_{i=1}^r)\mathcal{J}(x_k^{s_n})^\top\right)^{-1}\mathcal{J}(\{\tilde{x}_{i,k}^{s_n}\}_{i=1}^r)\Bigg)\nabla_x f_k(x_k^{s_n})\nonumber\\
&+\mathcal{J}(x_k^{s_n})^\top\left(\mathcal{J}(\{\tilde{x}_{i,k}^{s_n}\}_{i=1}^r)\mathcal{J}(x_k^{s_n})^\top\right)^{-1}\left(\frac{d_{k}-d_{k-1}}{s_n}\right)\nonumber\\
 := &g\left(\{\tilde{x}_{i,k}^{s_n}\}_{i=1}^r,x_k^{s_n},\left(\frac{d_{k}-d_{k-1}}{s_n}\right)\right)
\end{align}
Consider the continuously differentiable function $\bar x(t)$ that is introduced in Lemma~\ref{lem:cont}. The above equality together with~\eqref{limsupx} implies that
\begin{align}\label{ineq_limsup}
\lim\limits_{n\to\infty}\sup_{0\leq k\leq \lceil\frac{T}{s_n}\rceil}&\Bigg\|\dot {\bar{x}}(ks_n) 
- g\Big(\{\tilde{x}_{i,k}^{s_n}\}_{i=1}^r,x_k^{s_n},\Big(\frac{d_{k}-d_{k-1}}{s_n}\Big)\Big)\Bigg\| = 0
\end{align}
Therefore, one can write
\begin{align}\label{ineq_main}
&\lim\limits_{n\to\infty}\sup_{0\leq k\leq \lceil\frac{T}{s_n}\rceil}\Bigg\|\dot {\bar{x}}(ks_n)- g\left(\{\bar x(ks_n)\}_{i=1}^r,\bar x(ks_n),\dot d(ks_n\right)\Bigg\|\nonumber\\
\leq &\lim\limits_{n\to\infty}\sup_{0\leq k\leq \lceil\frac{T}{s_n}\rceil}\Bigg\|\dot{\bar{x}}(ks_n) - g\left(\{\tilde{x}_{i,k}^{s_n}\}_{i=1}^r,x^{s_n}_k,\left(\frac{d_{k}-d_{k-1}}{s_n}\right)\right)\Bigg\|\nonumber\\
& + \lim\limits_{n\to\infty}\sup_{0\leq k\leq \lceil\frac{T}{s_n}\rceil}\Bigg\|g\left(\{\bar x(ks_n)\}_{i=1}^r,\bar x(ks_n),\dot d(ks_n)\right) - g\left(\{\tilde{x}_{i,k}^{s_n}\}_{i=1}^r,x_k^{s_n},\left(\frac{d_{k}-d_{k-1}}{s_n}\right)\right)\Bigg\|
\end{align}
We present the following lemma.
\begin{lemma}\label{lem:lip}
	Given $(\{\bar x_i\}_{i=1}^r,\bar y,\bar z)$ with $ \left(\sum_{i=1}^{r}\|\bar x_i\|\right)+\|\bar y\|+\|\bar{z}\|\leq c_1$ for some $c_1>0$, suppose that\\
	 { $\sigma_{\min}\left(\mathcal{J}(\{\bar x_i\}_{i=1}^r)\mathcal{J}(\bar y)^\top\right)\geq c_2$ for some $c_2>0$.} Then, there exist universal constants $L,r>0$, such that $g(\{\bar x_i\}_{i=1}^r,\bar y,\bar z)$ is locally $L$-Lipschitz continuous in a ball $\mathcal{B} = \{(\{ x_i\}_{i=1}^r, y, z)\ |\ \left(\sum_{i=1}^{r}\|\bar x_i- x_i\|\right)+\|\bar y - y\|+\|\bar{z}-z\|\leq r \}$. 
\end{lemma}
\begin{proof}
	Due to the continuous differentiability of $\mathcal{J}(x)$ and Lemma~\ref{lem:fun}, it is easy to see that $r$ can be chosen such that $\sigma_{\min}\left(\mathcal{J}(\{x_i\}_{i=1}^r)\mathcal{J}(y)^\top\right)\geq c_2/2$ for every $(\{ x_i\}_{i=1}^r, y, z)\in\mathcal{B}(r)$. This observation, together with the definition of $g(\cdot,\cdot,\cdot)$ in~\eqref{discrete}, can be used to complete the proof. The details are omitted for brevity.
\end{proof}

{According to Lemma~\ref{lem:lip}, the function $g(\cdot,\cdot,\cdot)$ is locally Lipschitz continuous} on a ball with nonzero radius and centered at $\left(\{\tilde{x}_{i,k}^{s_n}\}_{i=1}^r,x_k^{s_n},\left(\frac{d_{k}-d_{k-1}}{s_n}\right)\right)$ for every $0\leq k\leq \lceil\frac{T}{\Delta t}\rceil$ and $n = 1,\dots,\infty$. This together with the definition of $\{\tilde{x}_{i,k}^{s_n}\}_{i=1}^r$, the differentiability of $d(t)$, and Lemma~\ref{lem:cont} implies that for sufficiently large $n$ (or, equivalently, for sufficiently small $s_n$), there exists a {Lipschitz constant $L$} such that
\begin{align}\label{ineq1}
&\Bigg\|g\left(\{\bar x(ks_n)\}_{i=1}^r,\bar x(ks_n),\dot d(ks_n)\right) - g\left(\{\tilde{x}_{i,k}^{s_n}\}_{i=1}^r,x_k^{s_n},\left(\frac{d_{k}-d_{k-1}}{s_n}\right)\right)\Bigg\|\nonumber\\
\leq& L\Bigg(\sum_{i=1}^r\|\bar x(ks_n)-\tilde{x}_{i,k}^{s_n}\|+\|\bar x(ks_n)- x_k^{s_n}\|+\left\|\dot d(ks_n)-\left(\frac{d_{k}-d_{k-1}}{s_n}\right)\right\|\Bigg)\nonumber\\
\leq &L\Bigg((r+1)\|\bar x(ks_n)- x_k^{s_n}\|+r\|\bar x((k-1)s_n)- x_{k-1}^{s_n}\|+r\|\bar x(ks_n)- \bar x((k-1)s_n)\|+\left\|\dot d(ks_n)-\left(\frac{d_{k}-d_{k-1}}{s_n}\right)\right\|\Bigg)\nonumber
\end{align}
{where we used the definition of $\{\tilde x_{i,k}^{s_n}\}_{i=1}^r$ and triangle inequality. According to Lemmas~\ref{lem:cont2} and~\ref{lem:cont}, the right-hand side of~\eqref{ineq1} converges to zero as $n\to\infty$. Therefore, combining~\eqref{ineq1} and~\eqref{ineq_limsup} with~\eqref{ineq_main} implies that} 
\begin{align}
&\lim\limits_{n\to\infty}\sup_{0\leq k\leq \lceil\frac{T}{s_n}\rceil}\Bigg\|\dot{\bar x}(ks_n) - g\left(\{\bar x(ks_n)\}_{i=1}^r,\bar x(ks_n),\dot d(ks_n\right)\Bigg\| = 0\nonumber
\end{align}
Furthermore, due to Lemma~\ref{lem:cont}, $\mathcal{J}(\bar x(t))$ is full-row rank at every $t\in[0,T]$ and therefore, $g(\{\bar x(t)\}_{i=1}^r, \bar x(t),\dot d(t))$ is continuous as a function of $t$ in $[0,T]$. Invoking Lemma~\ref{lem:equiv} then leads to
\begin{align}
\dot{\bar{x}}(t) = g(\{\bar x(t)\}_{i=1}^r, \bar x(t),\dot d(t))
\end{align}
at every $t\in[0,T]$. This shows that $\bar x:[0,T]\to\mathbb{R}^n$ is a solution to~\eqref{eq:ode}. Finally, due to Lemma~\ref{lem:cont}, we have $\sigma_{\min}(\mathcal{J}(\bar{x}(t)))\geq c$ for a universal constant $c>0$. Therefore, Lemma~\ref{lem:lip} can be used to verify the existence of an open and connected set $\mathcal{D}$ such that $g(\cdot,\cdot,\cdot)$ is locally $L$-Lipschitz continuous on $\mathcal{D}$ and $(\bar{x}(t),t)\in\mathcal{D}$ for every $t\in[0,T]$. Therefore, Theorem 2.2 in~\cite{coddington1955theory} can be used to show that $\bar{x}:[0,T]\to\mathbb{R}^n$ is the unique solution to~\eqref{eq:ode}. \qed

\subsection{Proof of Convergence} 

Next, we show the validity of the second statement in Theorem~\ref{thm:ode}.

\begin{lemma}[Backward Euler Iterations]\label{lem:back_euler}
	There exists a universal constant $\bar{t}$ such that for every $\Delta t\leq \bar{t}$, {there exists a sequence $\{y_k^{\Delta t}\}_{k=0}^{\lceil T/\Delta t\rceil}$ that satisfies the following statements:}
	\begin{itemize}
		\item[-] We have $y_0^{\Delta t} = x_0$ and
		{\begin{align}\label{back_euler}
		y_{k}^{\Delta t} = y_{k-1}^{\Delta t}+\Delta t\cdot g\left(\{y_{k}^{\Delta t}\}_{i=1}^r,y_k^{\Delta t},\dot{d}(s_k)\right)
		\end{align}}
		for  $k = 1,\dots,\lceil T/\Delta t\rceil$. 
		\item[-] There exists a universal constant $c_2>0$ such that $\|y_k^{\Delta t}-y_{k-1}^{\Delta t}\|\leq c_2\Delta t$ for  $k = 1,\dots,\lceil T/\Delta t\rceil$.
		\item[-] We have
		\begin{align}\label{backward}
		\lim_{\Delta t\to 0^+}\sup_{0\leq k\leq \lceil T/\Delta t\rceil}\|y_k^{\Delta t}-x(s_k)\| = 0
		\end{align}
		where $x:[0,T]\to\mathbb{R}^n$ is the unique solution to~\eqref{eq:ode}.
		\item[-] We have $\sigma_{\min}(\mathcal{J}(y_k^{\Delta t}))\geq c_1$ for some universal $c_1$ and every $k = 1,\dots,\lceil T/\Delta t\rceil$.
	\end{itemize}
\end{lemma}
\begin{proof}
	Note that~\eqref{back_euler} is the backward Euler iterations for~\eqref{eq:ode}~\cite{rudin1964principles}. Furthermore, we have already shown the existence of a continuously differentiable and uniformly bounded solution to~\eqref{eq:ode}. The proof of the first three statements is immediately followed by the classical results on convergence of the backward Euler method; see~\cite{rudin1964principles} for more details. To verify the correctness of the last statement, note that we have shown in the previous subsection that the function $\bar{x}:[0,T]\to\mathbb{R}^n$ introduced in Lemma~\ref{lem:cont} is indeed the unique solution to the proposed ODE and we have $\mathcal{J}(\bar{x}(t))\geq c$ for some universal $c>0$ and every $t\in[0,T]$. This together with~\eqref{backward} and Lemma~\ref{lem:fun} concludes the proof.
\end{proof}

{\it Proof of convergence:} The main idea behind the proof is to show
that, given any AKKT tuple $\left(x_0,\Delta t,\{x_k^{\Delta t}\}_{k=1}^{\lceil T/\Delta t\rceil}\right)$, we have
\begin{align}
\lim_{\Delta t\to 0^+}\sup_{0\leq k\leq \lceil T/\Delta t\rceil}\|y_k^{\Delta t}-x_k^{\Delta t}\| = 0
\end{align}
Establishing this equality together with Lemma~\ref{lem:back_euler} is enough to complete the proof. 

It is evident from~\eqref{discrete} that the AKKT tuple $\left(x_0,\Delta t,\{x_k^{\Delta t}\}_{k=1}^{\lceil T/\Delta t\rceil}\right)$ should satisfy
\begin{align}
x_k^{\Delta t} = x_{k-1}^{\Delta t}+\Delta tg\left(\{\tilde{x}_{i,k}^{\Delta t}\}_{i=1}^r,x_k^{\Delta t},\left(\frac{d_{k}-d_{k-1}}{\Delta t}\right)\right)
\end{align} 
where $\tilde{x}_{i;k}^{t_n} = (1-\alpha_i)x_k^{t_n}+\alpha_i x_{k-1}^{t_n}$ with $\alpha_i\in[0,1]$ for  $i = 1,\dots, n$. Combined with the first statement of Lemma~\ref{lem:back_euler}, this implies that
\begin{align}\label{error}
x_k^{\Delta t} \!-\! y_k^{\Delta t} =& x_{k-1}^{\Delta t} \!-\! y_{k-1}^{\Delta t}+ \Delta t\Bigg(g\left(\{\tilde{x}_{i,k}^{\Delta t}\}_{i=1}^r,x_k^{\Delta t},\left(\frac{d_{k}-d_{k-1}}{\Delta t}\right)\right)-g\left(\{y_{k}^{\Delta t}\}_{i=1}^r,y_k^{\Delta t},\dot{d}(s_k)\right)\Bigg)\\
 =& x_{k-1}^{\Delta t} \!-\! y_{k-1}^{\Delta t}+A+B \nonumber
\end{align}
where
\begin{subequations}
\begin{align}
A=&\Delta t\times\!\Bigg(g\left(\{\tilde{x}_{i,k}^{\Delta t}\}_{i=1}^r,x_k^{\Delta t},\left(\frac{d_{k}\!\!-\!d_{k-1}}{\Delta t}\right)\right)-\!g\left(\{y_{k-1}^{\Delta t}\}_{i=1}^r,y_{k-1}^{\Delta t},\dot{d}(s_k)\right)\Bigg),\\
B=&\Delta t\times \!\Bigg(g\left(\{y_{k-1}^{\Delta t}\}_{i=1}^r,y_{k-1}^{\Delta t},\dot{d}(s_k)\right)-g\left(\{y_{k}^{\Delta t}\}_{i=1}^r,y_{k}^{\Delta t},\dot{d}(s_k)\right)\Bigg).
\end{align}
\end{subequations}

Define $E_k = \|x_k^{\Delta t} \!-\! y_k^{\Delta t}\|$ as the error at time-step $k$. Note that, due to the Lemmas~\ref{lem:cont},~\ref{lem:back_euler}, and~\ref{lem:lip}, as well as the construction of $\{\tilde{x}_{i,k}^{\Delta t}\}_{i=1}^r$, there exist universal constants $L,\bar c,\bar{t}>0$ such that, {for every $\Delta t\leq\bar{t}$,} $g(\cdot,\cdot,\cdot)$ is locally $L$-Lipschitz continuous in the balls
\begin{align}\label{ball1}
&\mathcal{B}_1 = \Bigg\{(\{ x_i\}_{i=1}^r, y, z)\ \Bigg|\ \left(\sum_{i=1}^{r}\|\tilde{x}_{i,k}^{\Delta t}- x_i\|\right)+\|\tilde{x}_{k}^{\Delta t} - y\|+\left\|\left(\frac{d_{k}-d_{k-1}}{\Delta t}\right)-z\right\|\leq \bar{c} \Bigg\},
\end{align} 
and 
\begin{equation}
\label{ball2}
\begin{aligned}
&\mathcal{B}_2 = \Bigg\{(\{ x_i\}_{i=1}^r, y, z)\  \Bigg|\ \left(\sum_{i=1}^{r}\|y_k^{\Delta t}- x_i\|\right)+\|y_k^{\Delta t} - y\|+\left\|\dot{d}(s_k)-z\right\|\leq \bar{c} \Bigg\}.
\end{aligned}
\end{equation}
To simplify the notation, we denote $\left\|\left(\frac{d_k-d_{k-1}}{\Delta t}\right)-\dot{d}(s_k)\right\|$ as $D$. The following chain of inequalities will be useful in bounding the expression $A$ in~\eqref{error}:
\begin{align}\label{recur1}
&\left(\sum_{i=1}^r\left\|\tilde{x}_{i,k}^{\Delta t}-y_{k-1}^{\Delta t}\right\|\right)+\left\|x_k^{\Delta t}-y_{k-1}^{\Delta t}\right\|+D\nonumber\\
& \leq r\left\|x_{k-1}^{\Delta t}-y_{k-1}^{\Delta t}\right\| + (r+1)\left\|x_k^{\Delta t}-y_{k-1}^{\Delta t}\right\|+D\nonumber\\
&\leq r\left\|x_{k-1}^{\Delta t}-y_{k-1}^{\Delta t}\right\| + (r+1)\left\|x_k^{\Delta t}-x_{k-1}^{\Delta t}\right\|\nonumber\\
&\hspace{2mm}+(r+1)\left\|x_{k-1}^{\Delta t}-y_{k-1}^{\Delta t}\right\|+D\nonumber\\
& = (2r+1) E_{k-1}+(r+1)\left\|x_k^{\Delta t}-x_{k-1}^{\Delta t}\right\|\nonumber+D\nonumber\\
&\leq (2r+1) E_{k-1}+(r+1)c_1\Delta t + c_2\Delta t^2\nonumber\\
&\leq (2r+1) E_{k-1}+((r+1)c_1+c_2)\Delta t
\end{align}
provided that $\Delta t\leq \bar{t}_1$, where  $\bar{t}_1,c_1,c_2>0$ are constants. Note that the last two inequalities are due to Lemma~\ref{lem:cont2} and the twice differentiability of $d(t)$. 

Subsequently, the next inequality will be used to bound the expression $B$ in~\eqref{error}. In particular, Lemma~\ref{lem:back_euler} can be used to show the existence of constants $c_3,\bar{t}_2>0$ such that
\begin{align}\label{recur2}
(r+1)\|y_{k-1}^{\Delta t}-y_{k}^{\Delta t}\|\leq c_3\Delta t
\end{align}
provided that $\Delta t\leq \bar{t}_2$. Given the inequalities~\eqref{recur1} and~\eqref{recur2}, we prove the validity of~\eqref{converge1} by proving the following statements:
\begin{itemize}
	\item[1.] There exists a universal constant $\bar{t}_3$ such that for every $\Delta t\leq\bar{t}_3$ and $k = 0,\dots, T/\Delta t$,~\eqref{recur1} and~\eqref{recur2} will be upper bounded by $\bar{c}$ which is defined as the radius of the balls~\eqref{ball1} and~\eqref{ball2}. This together with the locally $L$-Lipschitz continuity of $g(\cdot,\cdot,\cdot)$ within the balls $\mathcal{B}_1$ and $\mathcal{B}_2$ leads to
	\begin{subequations}
		\begin{align}
		\|A\|&\leq (2r+1)L\Delta tE_{k-1}+((r+1)c_1+c_2)L\Delta t^2\\
		\|B\|&\leq  c_3L\Delta t^2
		\end{align}
	\end{subequations}
	Combining these inequalities with~\eqref{error} results in the following recursive inequality:
	\begin{align}\label{eq2}
	E_k \leq &(1+(2r+1)L\Delta t)E_{k-1} + \left((r+1)c_1+c_2+c_3\right)L\Delta t^2
	\end{align}
	\item[2.] We have $\lim_{\Delta t\to0^+}\sup_{0\leq k\leq T/\Delta t}E_k = 0$.
\end{itemize}
We prove the first statement using an inductive argument on $k$. In particular, we show that if the following inequality holds 
\begin{align}\label{Delta}
\Delta t&\leq \min\Bigg\{\bar{t}_1,\bar{t}_2,\sqrt{\frac{\bar{c}}{{\left((r+1)c_1+c_2+c_3\right)L}}},\frac{(2r+1)\bar{c}}{((r+1)c_1+c_2+c_3)(e^{(2r+1)TL}-1)}\Bigg\}\nonumber\\
&=\bar{t}_3
\end{align}
then ~\eqref{recur1} and~\eqref{recur2} remain in the balls $\mathcal{B}_1$ and $\mathcal{B}_2$, respectively and hence,~\eqref{eq2} holds for  $k=0,\dots,T/\Delta t$.

{\bf Base case: $\mathbf{k=1}$.} Note that in this case, $E_0 = 0$ and therefore, based on~\eqref{Delta}, we have $\Delta t\leq \bar{t}_1$ and $\Delta t\leq \bar{t}_2$. This implies that both~\eqref{recur1} and~\eqref{recur2} are upper bounded by $\bar{c}$ and, based on~\eqref{eq2}, we have
\begin{align}
\label{eq2_1}
E_1 &\leq (1+(2r+1)L\Delta t)E_{0} + \left((r+1)c_1+c_2+c_3\right)L\Delta t^2 \nonumber\\
&= \left((r+1)c_1+c_2+c_3\right)L\Delta t^2 \leq \bar{c}
\end{align}
where the last inequality is due to~\eqref{Delta}.

{\bf Inductive step.} Suppose that we have
\begin{subequations}
\begin{align}
&(2r+1) E_{k-1}+((r+1)c_1+c_2)\Delta t\leq \bar{c}\\
&c_3\Delta t\leq \bar{c}
\end{align}
\end{subequations}
for  $k = 0,\dots, m-1$. This implies that~\eqref{eq2} holds for  $k = 1,\dots, m$. With some algebra, one can verify that
\begin{align}
&E_m\leq \left((r+1)c_1+c_2+c_3\right)L\Delta t^2\sum_{i=0}^{m-1}(1+(2r+1)L\Delta t)^i\nonumber\\
&\leq \left((r+1)c_1+c_2+c_3\right)L\Delta t^2\cdot \frac{(1+(2r+1)L\Delta t)^m-1}{(2r+1)L\Delta t}\nonumber\\
&\leq \frac{(r+1)c_1+c_2+c_3}{2r+1}\left((1+(2r+1)L\Delta t)^{T/\Delta t}-1\right)\Delta t\nonumber\\
&\leq \frac{(r+1)c_1+c_2+c_3}{2r+1}\left(e^{(2r+1)LT}-1\right)\Delta t\leq\bar{c}
\end{align}
which completes the proof of the first statement. {To prove the second statement, note that the above analysis leads to}
\begin{align*}
\sup_{0\leq k\leq T/\Delta t}E_k\leq \frac{(r+1)c_1+c_2+c_3}{2r+1}\left(e^{(2r+1)LT}-1\right)\Delta t
\end{align*}
assuming that $\Delta t\leq \bar{t}_3$. Due to the fact that $\bar{t}_3>0$ and is independent of $\Delta t$, we have
\begin{align}
\lim\limits_{\Delta t\to 0^+}\sup_{0\leq k\leq T/\Delta t}E_k = 0 
\end{align}
thereby completing the proof of the convergence.\qed

\section{Properties of System's Jacobian}
\label{sec:sys_Jac}
{
In this section, we additionally assume that the objective function $f(x,t)$ is twice continuously differentiable in $x$. For the constraint functions $h=(h_1,h_2,\ldots,h_m)$, the corresponding Hessian matrices $H_1,H_2,\ldots,H_m \in \mathbb{R}^{n\times n}$ are the second partial derivative of $h$ with respect to $x$. The second-order derivative operator of $h$, denoted by $H$, is now regarded as the $m$-tuple $H=(H_1,\ldots,H_m)$. For $\mu \in \mathbb{R}^m$ and $x\in\mathbb{R}^n$,  $\mu H$ denotes $\mu_1 H_1+\ldots+\mu_m H_m$ and $x^\top H x$ denotes $x^\top H_1 x+\ldots+x^\top H_m x$. For $M_1, M_2\in \mathbb{R}^{n\times n}$, $M_1 H M_2 x$ denotes $[M_1 H_1 M_2 x, \ldots, M_1 H_m M_2 x]$.
In addition, we have the identity $\mu x^\top H x=x^\top \mu H x$. \\
\indent Consider the time-invariant optimization problem:
\begin{equation} 
\inf_{x\in \mathbb{R}^n} ~ f(x) ~~~\text{s.t.}~~~ h(x) = d
\label{eq:invariantoptimization}
\end{equation}
where $h(x) = [h_1(x),\hdots,h_m(x)]^T$ and $d = [d_1,\hdots,d_m]^T$. The corresponding ODE is given by
\begin{align} 
\dot{x} = -\frac{1}{\alpha} \left[I-\mathcal{J}(x)^\top(\mathcal{J}(x)\mathcal{J}(x)^\top)^{-1}\mathcal{J}(x)\right]\nabla f(x).
\label{eq:odeinvariant}
\end{align}
Let $z$ be a local minimum of \eqref{eq:invariantoptimization} satisfying the first-order necessary and second-order sufficient optimality conditions:
\begin{subequations}
\begin{align}
 &h(z)=d, \ \mathcal{J}(z)\mathcal{J}(z)^\top \text{ is invertible}\\
& \nabla f(z)+\mu\mathcal{J}(z)= 0,\ w^\top \left(\nabla^2f(z)+\mu\mathcal{H}(z)\right)w>0   \label{sosc}
\end{align}
\end{subequations}
for some $\mu \in \mathbb{R}^m$ and every nonzero vector $w$ such that $\mathcal{J}(z)^\top w=0$. Note that $z$ is an equilibrium point of the system~\eqref{eq:odeinvariant}. Let the right-hand side of~\eqref{eq:odeinvariant} be denoted by $p(x)$: 
\begin{equation}\label{ode_single2}
p(x):= -\frac{1}{\alpha} \mathcal{P}(x) \nabla_x f(x)
\end{equation}
where $\mathcal{P}(x)= I-\mathcal{J}(x)^\top(\mathcal{J}(x)\mathcal{J}(x)^\top)^{-1}\mathcal{J}(x)$ and let $\mathcal{J}_p(z)$ denote the Jacobian of $p(x)$.
\begin{theorem}\label{lem:psd}
It holds that
\begin{align}
\label{eq_jac_123}
\mathcal{J}_p(z)=-\frac{1}{\alpha}\left(\nabla^2f(z)+\mu\mathcal{H}(z)\right)\mathcal{P}(z).
\end{align}
Moreover, $\mathcal{J}_p(z)$ has $n-m$ eigenvalues with negative real parts and $m$ zero eigenvalues.
\end{theorem}
\begin{proof}
The equation \eqref{eq_jac_123} follows from Corollary 1 in \cite{luenberger1972gradient}.
To study the eigenvalues of $\mathcal{J}_p(z) $, note that $\mathcal{J}_p(z) \mathcal{J}(z)^\top=0$. Therefore, $\mathcal{J}_p(z)$ has at least $m$ zero eigenvalues. Let $w\in\mathbb R^n$ be an arbitrary nonzero vector in the tangent plane of the manifold $\{x\ :\ h(x)=d\}$ at the point $x=z$. This means that $\mathcal{J}(z)^\top w=0$. On the other hand, the second-order sufficient optimality condition states that
$
w^\top \left(\nabla^2f(z)+\mu H(z)\right)w>0
$.
Therefore, we have
$
w^\top \Omega w>0
$,
where
\begin{align}
\Omega=&\mathcal{P}(z)\left(\nabla^2f(z)+\mu H(z)\right)\mathcal{P}(z).
\end{align}
Since $\mathcal{J}(z)$ is in the null space of the symmetric matrix $\Omega$ and $w^\top \Omega w>0$ for every $w$ that is orthogonal of $\mathcal{J}(z)$, it can be concluded that $\Omega$ has $n-m$ eigenvalues with positive real parts. 
On the other hands, the eigenvalues of $\Omega$ are the same of the eigenvalues of the matrix
\begin{align}
\left(\nabla^2f(z)+\mu H(z)\right)\mathcal{P}^2(x)=\left(\nabla^2f(z)+\mu H(z)\right)\mathcal{P}(z)
\end{align}
which is the identical to $-\alpha \mathcal{J}_p(z)$.~\end{proof}
As shown above, the eigenvalues of the Jacobian only have non-positive real parts. This explains why spurious solutions of a time-invariant optimization problem cannot be escaped using gradient-based methods, such as the ODE \eqref{eq:odeinvariant}. Now, consider its time-varying counterpart problem \eqref{eq:problem} and associated ODE \eqref{eq:ode}.
Let $z(t):[0,T]\rightarrow \mathbb R^n$ be a local solution of \eqref{eq:problem}
that satisfies the first-order necessary and second-order sufficient optimality conditions for all $t\in[0,T]$. Let $\mu(t)$ denote the corresponding Lagrange multiplier and $\mathcal{Q}(z(t))$ denote $\mathcal{J}(z(t))^\top(\mathcal{J}(z(t))\mathcal{J}(z(t))^\top)^{-1}$. Since $z(t)$ is generally not the solution of the ODE \eqref{eq:ode}, we make a change of variables $e(t)=x(t)-z(t)$ to measure the distance between $x(t)$ and $z(t)$. Then, the ODE \eqref{eq:ode} can be rewritten as 
\begin{equation} \label{eq: ODE change}
\dot{e}(t)=-\frac{1}{\alpha} \eta(e(t)+z(t),t)+\theta(e(t)+z(t))\dot{d}-\dot{z}(t)
\end{equation}
Let $\mathcal{J}_q(z(t))$ denote the Jacobian of the right-hand side of \eqref{eq: ODE change} at the point $e(t)=0$.
By taking the first-order approximation of \eqref{eq: ODE change} around $z(t)$, we have
\begin{equation} \label{eq: jacobian first order}
\dot{e}(t)=\mathcal{J}_q(z(t))e(t)+O(e^2(t)) -\dot{z}(t).
\end{equation}
\begin{theorem}\label{lem:psd2}
It holds that
\begin{equation}
\mathcal{J}_q(z(t)) =K_1(t)+K_2(t)
\end{equation}
where
\begin{subequations}
\begin{align}
K_1(t)= &-\frac{1}{\alpha}\left(\nabla^2f(z(t))+\mu(t)\mathcal{H}(z(t))\right)\mathcal{P}(z(t)),\\
K_2(t)=& \Big(P(z(t))H(z(t)) \left(\mathcal{J}(z(t)) \mathcal{J}(z(t))^\top \right)^{-1}-\mathcal{Q}(z(t))H(z(t))\mathcal{Q}(z(t)) \Big)\dot{d}(t).
\end{align}
\end{subequations}
\end{theorem}
\begin{proof}
The computation of $K_1(t)$ is similar to that of Theorem \ref{lem:psd}. Because of the tensor nature of $H$ it is convenient to differentiate with respect to each component separately. For the component $z_1(t)$, we have
\begin{align*}
&\frac{d}{d z_1(t)} \mathcal{Q}(z(t))\dot{d}(t)\\
=&H_1(z(t)) \left(\mathcal{J}(z(t)) \mathcal{J}(z(t))^\top \right)^{-1} \dot{d}(t) \\
&- \mathcal{J}(z(t))^\top\left(\mathcal{J}(z(t)) \mathcal{J}(z(t))^\top \right)^{-1}  \Big( H_1(z(t))\mathcal{J}(z(t))^\top + \mathcal{J}(z(t)) H_1(z(t))  \Big)\left(\mathcal{J}(z(t)) \mathcal{J}(z(t))^\top \right)\dot{d}(t)\\
=& \Big(P(z(t))H_1(z(t)) \left(\mathcal{J}(z(t)) \mathcal{J}(z(t))^\top \right)^{-1}-\mathcal{Q}(z(t))H_1(z(t))\mathcal{Q}(z(t)) \Big)\dot{d}(t).
\end{align*}
Similar expressions apply to derivatives with respect to other components. These columns can be combined into the matrix 
\[\Big[\frac{d}{d z_1(t)} \mathcal{Q}(z(t)),\ldots,\frac{d}{d z_n(t)} \mathcal{Q}(z(t))\Big]\dot{d}(t).\]
This matrix is $K_2(t)$.
\end{proof}
Notice that $K_1(t)$ has only eigenvalues with non-positive reals (due to Theorem \ref{lem:psd}) but $K_2(t)$ may have eigenvalues with positive reals depending on the time-variation. Thus, the time variation could potentially make the linear system $\dot{\bar e}(t)=\mathcal{J}_q(z(t)) \bar e(t)$ unstable. If $O(e^2(t)) -\dot{z}(t)$ is not large, we may expect that the solution of \eqref{eq: jacobian first order} will behavior similarly to  $\dot{\bar e}(t)=\mathcal{J}_q(z(t)) \bar e(t)$ and cannot stay around the point $0$. Thus, the time-variation may provide the opportunity to escape the spurious local trajectory $z(t)$. Note that the linearization does not always provide a concrete answer for time-varying ODEs, but this result offers an insight into how the data variation changes the eigenvalues of the Jacobian along a trajectory close to a KKT trajectory.
}
	\begin{figure*}
      \centering
      \subcaptionbox{Inequalities in function of $\alpha,\beta$ guaranteeing absence of spurious trajectories.\label{fig:ineq}}
        {  \includegraphics[width=0.38\linewidth]{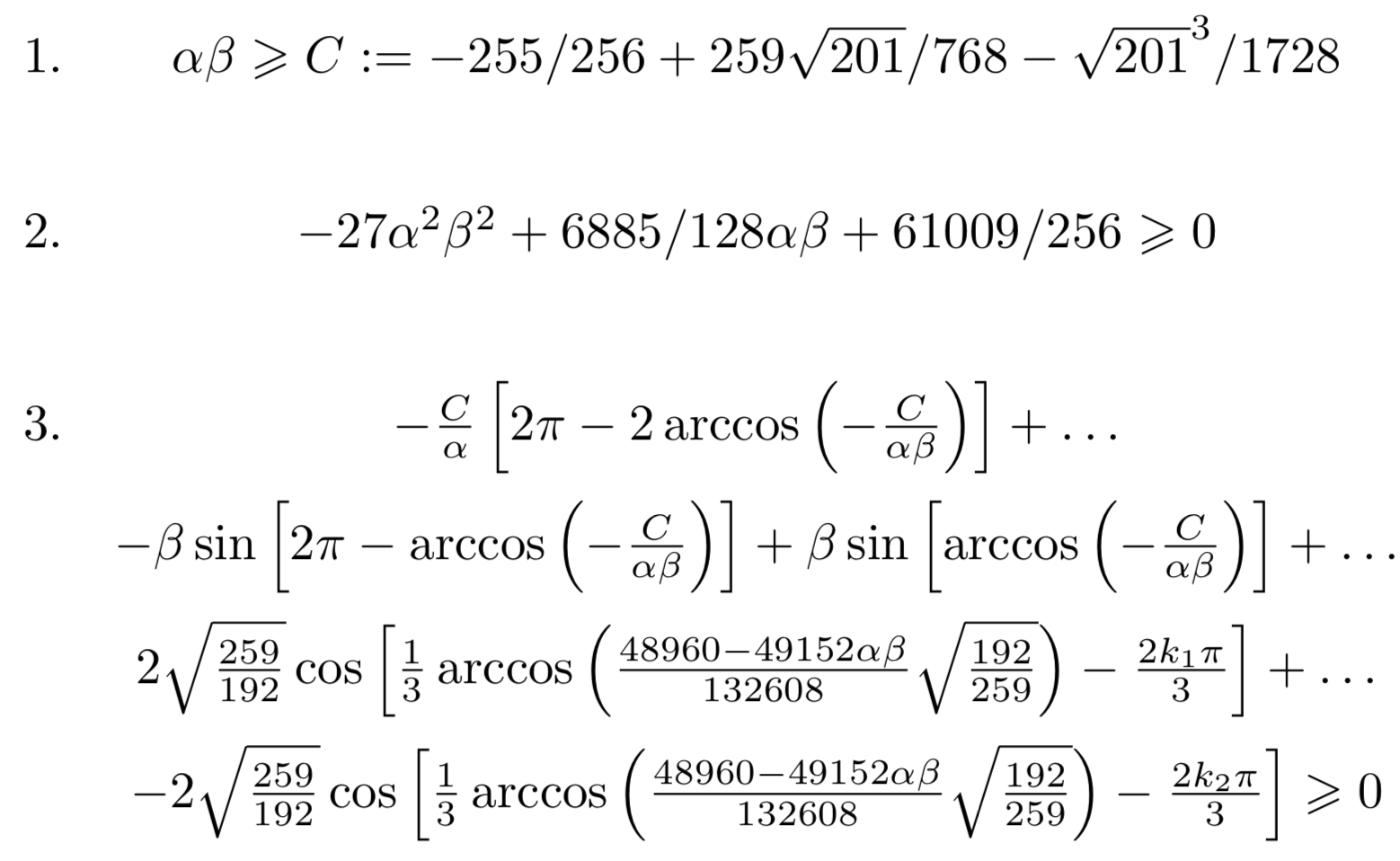}}\hspace{3mm}
      \subcaptionbox{Sufficient condition in blue in function of $\alpha,\beta$ for absence of spurious trajectories.\label{fig:graph}}
        { \includegraphics[width=0.39\linewidth]{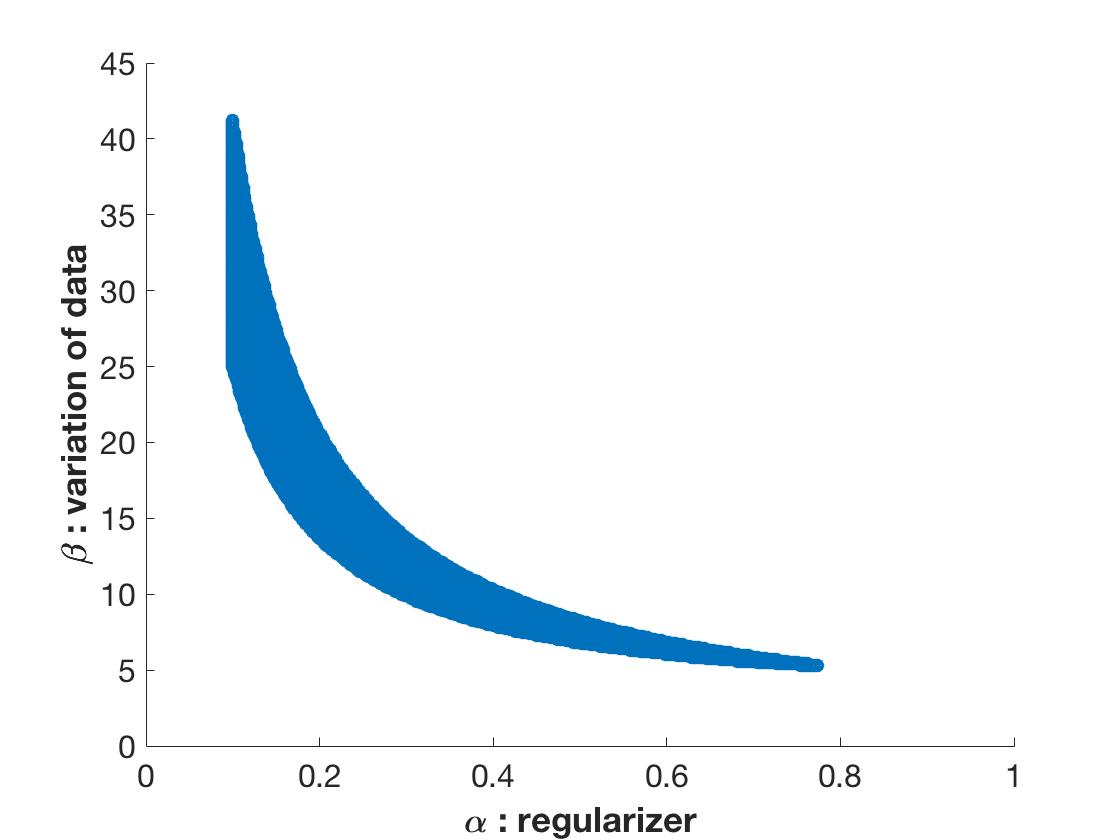}}
   \subcaptionbox{Non-spurious trajectory for $\alpha = 0.4$ and $\beta = 10$.\label{fig:non-spurious}}
        {  \includegraphics[width=0.38\linewidth]{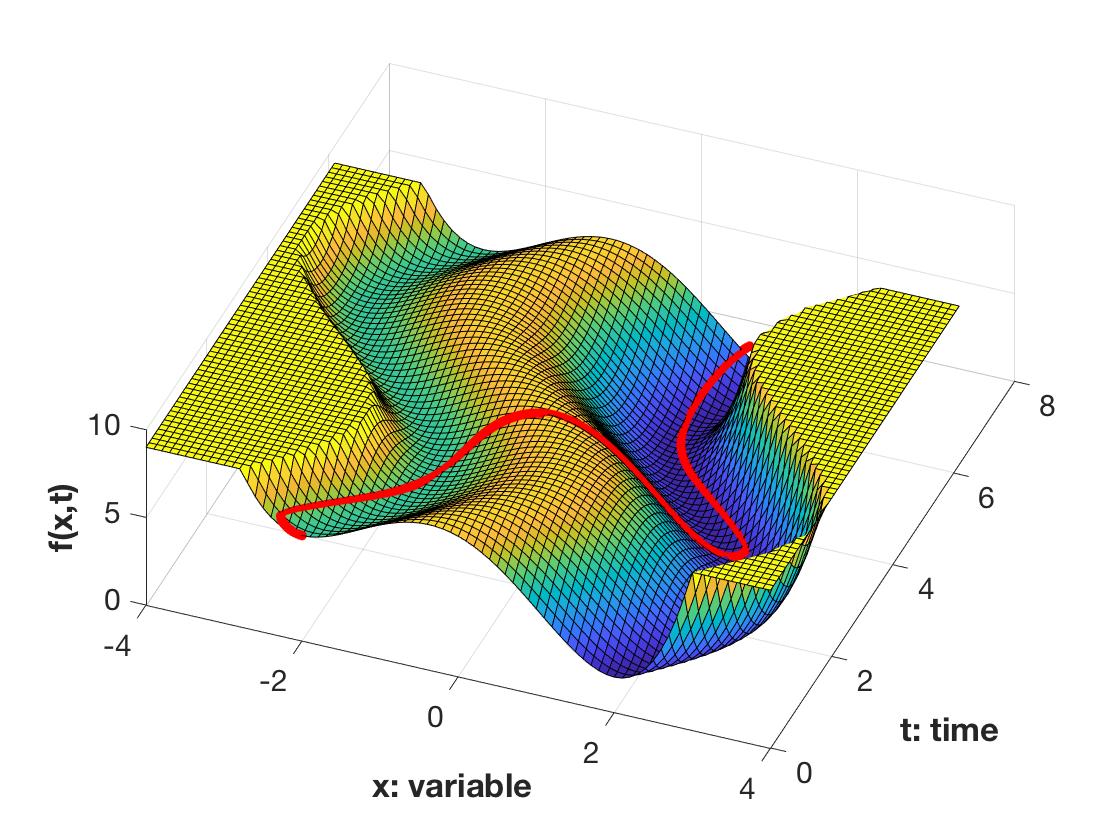}}\hspace{3mm}
      \subcaptionbox{Spurious trajectory for $\alpha = 0.2$ and $\beta = 5$.\label{fig:spurious}}
        { \includegraphics[width=0.39\linewidth]{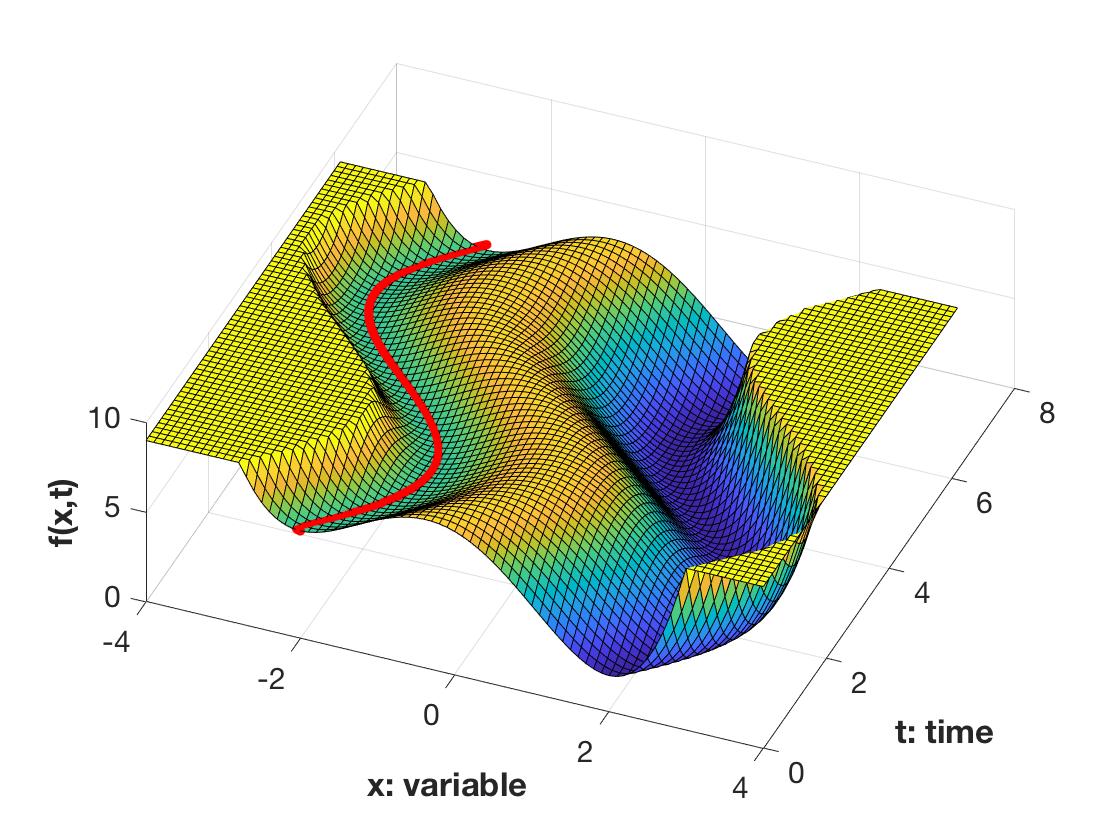}}
      \caption{Analysis of Example \ref{exp1}.}\vspace{-2mm}
    \end{figure*}

\section{Conclusion}\label{sec:Conclusion}
In this work, we study the landscape of time-varying nonconvex optimization problems. We introduce the notion of spurious local trajectory as a counterpart to the notion of spurious local minima in the time-invariant optimization. The key insight to this new notion is the fact that a regularized version of the time-varying optimization problem is naturally endowed with an ordinary differential equation (ODE) at its limit. This close interplay enables us to study the solutions of this ODE to certify the absence of the spurious local trajectories in the problem. Through different case studies and theoretical results, we show that a time-varying optimization can have multiple spurious local minima, and yet its landscape can be free of spurious local trajectories. We further show that the variation of the landscape over time is the main reason behind the absence of spurious local trajectories.  We also study the Jacobian of the ODE along a local minimum trajectory and show how its eigenvalues  change in response to the data variation. {Avenues for future work include the extension of the notion of spurious local trajectories to time-varying optimization over an infinite-time horizon. Furthermore, it would be worthwhile to
derive necessary and sufficient conditions for the absence of spurious local trajectories in more general
settings. }
\appendix
{\begin{lemma}\label{l_app}
We have $\|x_k^{\Delta t}-x_{k-1}^{\Delta t}\| = O(\sqrt{\Delta t})$ for every $k=0,\dots,{\lceil T/\Delta t\rceil}$.
\end{lemma}
\begin{proof}
Note that $f(x,t)$ is uniformly bounded from below. Furthermore, for every AKKT tuple $\left(x_0,\Delta t,\{x_k^{\Delta t}\}_{k=0}^{{\lceil T/\Delta t\rceil}}\right)$, the sequence $\{x_k^{\Delta t}\}_{k=0}^{{\lceil T/\Delta t\rceil}}$ is assumed to be uniformly bounded. {This together with Assumption~\ref{assum11} implies that 
	\begin{align}
	    f(x_k^{\Delta t},t_k)+\frac{\alpha}{2\Delta t}\|x_k^{\Delta t}-x_{k-1}^{\Delta t}\|^2\leq R
	\end{align}
	for some $R<\infty$. Since $f(x_k^{\Delta t},t_k)$ is assumed to be uniformly bounded from below, this leads to $\frac{\alpha}{2\Delta t}\|x_k^{\Delta t}-x_{k-1}^{\Delta t}\|^2\leq R'$ for some $R'<\infty$, which in turn yields $\|x_k^{\Delta t}-x_{k-1}^{\Delta t}\| = O(\sqrt{\Delta t})$.}
\end{proof}}
{{\it Proof of Lemma~\ref{lem:cont2}.} Due to Lemma~\ref{l_app} and the fact that $\mathcal{J}(x)$ is continuously differentiable, one can invoke Lemma~\ref{lem:fun} to show that there exist constants $\bar{t},c_1,c_2>0$ such that the following statements hold, provided that $\Delta t\leq \bar{t}$:
\begin{itemize}
	\item[1.] 
	\begin{sloppypar}
		Consider a sequence $\{\tilde{x}_{i,k}^{\Delta t}\}_{i=1}^r$ constructed similar to~\eqref{jacob}. Due to Assumption~\ref{assum2} and Lemma~\ref{l_app}, it can be verified that there exist $\bar{t},c_1>0$ such that $\sigma_{\min}(\mathcal{J}(\{\tilde{x}_{i,k}^{\Delta t}\}_{i=1}^r)\mathcal{J}(x_k^{\Delta t})^\top)\geq c_1$ for all $\Delta t\leq\bar{t}$. This implies that the function $g\left(\{\tilde{x}_{i,k}^{\Delta t}\}_{i=1}^r,x_k^{\Delta t},\left(\frac{d_{k}-d_{k-1}}{\Delta t}\right)\right)$ introduced in~\eqref{discrete} is well-defined and continuous for all $\Delta t\leq \bar{t}$.
	\end{sloppypar}
	\item[2.] Assumption~\ref{assum1} and twice differentiability of $d$ with respect to $t$ imply that $\left\{\{\tilde{x}_{i,k}^{\Delta t}\}_{i=1}^r,x_k^{\Delta t}\right\}$ and $\left(\frac{d_{k}-d_{k-1}}{\Delta t}\right)$ belong to a compact set. Combined with the continuity of $g(\cdot)$, this implies that
	\begin{align}\label{eqg}
	\left\|g\left(\{\tilde{x}_{i,k}^{\Delta t}\}_{i=1}^r,x_k^{\Delta t},\left(\frac{d_{k}-d_{k-1}}{\Delta t}\right)\right)\right\|\leq c_2
	\end{align}
	for some $c_2>0$.
	\item[3.] Similar to~\eqref{discrete}, one can verify that the following equality holds:
\begin{align*} 
\frac{x_k^{\Delta t}-x_{k-1}^{\Delta t}}{\Delta t} = g\left(\{\tilde{x}_{i,k}^{\Delta t}\}_{i=1}^r,x_k^{\Delta t},\left(\frac{d_{k}-d_{k-1}}{\Delta t}\right)\right)
\end{align*}
\end{itemize}
Combined with~\eqref{eqg}, this implies that $\|x_k^{\Delta t}-x_{k-1}^{\Delta t}\|\leq c_2\Delta t$ and the proof is complete.$\hfill\square$}

{\it Proof of Lemma~\ref{lem:cont}.} Consider a sequence $\{s_n\}_{n=1}^\infty$ such that $s_n>0$ and $\lim_{n\to\infty}s_n = 0$. Furthermore, without loss of generality, we assume that $T/s_n$ is a natural number for every $n=1,\dots,\infty$.
Given any $n$, consider a AKKT tuple $(x_0,s_n,\{x_k^{s_n}\}_{k=0}^\infty)$ and define a vector-valued function $\tilde{x}_{s_n} : [0,{T}]\to \mathbb{R}^n$ whose $i^{\mathrm{th}}$ element is the spline interpolation of the $i^{\mathrm{th}}$ elements of the vectors $\{x_0^{s_n},x_1^{s_n},\dots, x_{T/s_n}^{s_n}\}$. Notice that this interpolation can be made in such a way that $\tilde{x}_{s_n}$ is continuously differentiable.

We prove this lemma by showing that {there exist} a continuously differentiable function $\bar{x}$ and a subsequence $\{\tilde{x}_{t_{n_r}}\}_{r=1}^{\infty}$ of $\{\tilde{x}_{s_n}\}_{n=1}^{\infty}$ such that $\{\tilde{x}_{t_{n_r}}\}_{r=1}^{\infty}$ and $\{\dot{\tilde{x}}_{t_{n_r}}\}_{r=1}^{\infty}$ converge uniformly to $\bar{x}$ and $\dot{\bar{x}}$, respectively.
Note that $\tilde{x}_{s_n}$ is continuous for  $n = 1,\dots,\infty$, due to Lemma~\ref{lem:cont2}. Consider the class of functions $\mathcal{X} = \{\tilde{x}_{s_n}\ |\ n = 1,\dots,\infty\}$. $\mathcal{X}$ is uniformly bounded (due to Assumption~\ref{assum2}) and equicontinuous. Therefore, the Arzel\`a-Ascoli theorem can be invoked to show the existence of a uniformly convergent subsequence $\{\tilde{x}_{t_{n_k}}\}_{k=1}^{\infty}$. Let $\bar{x}:[0,T]\to \mathbb{R}^n$ be the limit of $\{\tilde{x}_{t_{n_k}}\}_{k=1}^{\infty}$. Now, consider the sequence $\{\dot{\tilde{x}}_{t_{n_k}}\}_{k=1}^{\infty}$. Notice that, due to the construction, $\{\dot{\tilde{x}}_{t_{n_k}}\}_{k=1}^{\infty}$ is continuous. Consider the class of functions $\bar{\mathcal{X}} = \{\dot{\tilde{x}}_{t_{n_k}}\ |\ k = 1,\dots,\infty\}$. Similar to the previous case, $\bar{\mathcal{X}}$ is uniformly bounded and equicontinuous. Therefore, another application of Arzel\`a - Ascoli theorem implies that $\{\dot{\tilde{x}}_{t_{n_k}}\}_{k=1}^{\infty}$ has a subsequence $\{\dot{\tilde{x}}_{t_{n_r}}\}_{r=1}^{\infty}$ that converges uniformly to a function $y:[0,T]\to \mathbb{R}^n$. Since $\{n_r\}_{r=1}^\infty\subseteq\{n_k\}_{k=1}^\infty$, we have that $\{{\tilde{x}}_{t_{n_r}}\}_{r=1}^{\infty}$ converges uniformly to $\bar x$. Therefore, due to Theorem 7.17 of~\cite{rudin1964principles}, we have $\dot{\bar{x}} = y$.

Finally, recall that $\{x_k^{s_n}\}_{n=1}^\infty$ is uniformly bounded and there exists a universal constant $c$ such that $\mathcal{J}(x_k^{s_n})\geq c$ for  $k=0,\dots,T/s_n$ and $n=1,\dots,\infty$. This implies that the function sequence $\{{\tilde{x}}_{t_{n_r}}\}_{r=1}^{\infty}$ is also uniformly bounded and since they converge uniformly to $\bar{x}$, one can invoke Lemma~\ref{lem:fun} to verify the existence of a universal $c'>0$ such that $c\geq c'$ and $\mathcal{J}(\bar{x}(t))\geq c'$ for every $t\in[0,T]$. $\hfill\square$

\bibliography{bib}{}
\bibliographystyle{IEEEtran}

\footnotesize

\noindent {\bf Salar Fattahi}
is an assistant professor at University of Michigan. He has received the INFORMS Data Mining Best Paper Award and been a best paper finalist for ACC 2018.

\noindent{\bf C\'edric Josz} is an assistant professor at Columbia University. He received the 2016 Best Paper Award in Springer Optimization Letters and was a finalist of the competition for best PhD thesis of 2017 organized by French Agency for Mathematics in Interaction with Industry and Society.

\noindent{\bf Yuhao Ding}
is a Ph.D. candidate at UC Berkeley. His research interests are nonlinear optimization and machine learning.

\noindent{\bf Reza Mohammadi}
is a postdoctoral scholar  at UC Berkeley. He obtained his PhD from MIT. His research interests are statistical learning, nonlinear optimization and energy.

\noindent{\bf Javad Lavaei}
is an associate professor at UC Berkeley. He is an associate editor for  IEEE Transactions on Automatic Control, IEEE Transactions on Smart Grid and  IEEE Control Systems Letters.

\noindent{\bf Somayeh Sojoudi}
is an assistant professor at UC Berkeley. She is an associate editor for IEEE Transactions on Smart Grid, IEEE Access, and Systems \& Control Letters.

\end{document}